\documentclass[12pt]{amsart}
\usepackage{epsfig,comment}
\usepackage{mathtools}
\usepackage[subnum]{cases}
\usepackage{enumerate}
\usepackage{bbm,bm}
\usepackage{amsmath,amssymb,mathrsfs,amsthm}
\usepackage{color}
\usepackage[bookmarks=true,
bookmarksnumbered=true, breaklinks=true,
pdfstartview=FitH, hyperfigures=false,
plainpages=false, naturalnames=true,
colorlinks=true,pagebackref=true,
pdfpagelabels,colorlinks=true,
	linkcolor=black,
	citecolor=black,
	filecolor=black,
	urlcolor=black]{hyperref}
\usepackage{wasysym}
\usepackage[hang]{footmisc} 
\usepackage{verbatim}

\newcommand{\E}{\mathbb{E}}

\newcommand{\N}{\mathbb{N}}
\renewcommand{\P}{\mathbb{P}}
\newcommand{\R}{\mathbb{R}}
\renewcommand{\S}{\mathbb{S}}

\newcommand{\NN}{N^{-\frac{2}{n-1}}}

\newcommand{\cK}{\mathcal{K}}

\DeclareMathOperator{\vol}{Vol}

\DeclareMathOperator{\Gr}{Gr}

\DeclareMathOperator{\dist}{dist}

\DeclareMathOperator{\interior}{int}

\newcommand{\dint}{\mathrm{d}}
   
\renewcommand{\phi}{\varphi}

\newtheorem {theorem}{Theorem}[section]
\newtheorem {proposition}[theorem]{Proposition}
\newtheorem {lemma}[theorem]{Lemma}

\newtheorem {question}[theorem]{Open Question}

\theoremstyle{definition}
\newtheorem {definition}{Definition}[section]
\newtheorem {remark}[theorem]{Remark}

\newcommand{\bbinom}[2]{\left[\!\!\begin{array}{c}#1\\#2\end{array}\!\!\right]}
\newcommand{\tbbinom}[2]{\scalebox{0.75}{$\left[\!\!\renewcommand{\arraystretch}{0.6}\begin{array}{c}#1\\#2\end{array}\!\!\right]$}}

\usepackage{accents}
\DeclareMathSymbol{\widetildesym}{\mathord}{largesymbols}{"65}

\title[An intrinsic volume metric for the class of convex bodies in $\mathbb{R}^n$]{An intrinsic volume metric\\ for the class of convex bodies in $\mathbb{R}^n$}

\author{Florian Besau and Steven Hoehner}
\date{\today}

\begin{document}

\setcounter{footnote}{0}
\maketitle

\begin{abstract}\noindent
A new intrinsic volume metric is introduced for the class of convex bodies in $\mathbb{R}^n$. As an application, an  inequality is proved for the asymptotic best approximation of the Euclidean unit ball by arbitrarily positioned polytopes with a restricted number of vertices under this metric. This result improves the best known estimate, and shows that dropping the restriction that the polytope is contained in the ball or vice versa improves the estimate by at least a factor of dimension. The same phenomenon has already been observed in the special cases of volume, surface area and mean width approximation of the ball.
\end{abstract}

\renewcommand{\thefootnote}{}

\renewcommand{\thefootnote}{\arabic{footnote}}
\setcounter{footnote}{0}

\section{Introduction and main results}

The intrinsic volumes  $V_0(K),V_1(K),\ldots,V_n(K)$ of a convex body $K$ in $\R^n$ are defined as the coefficients in \emph{Steiner's formula} for the volume of the outer parallel body
\begin{equation}
    \vol_n(K+\varepsilon B_n)=\sum_{j=0}^n \varepsilon^{n-j}\vol_{n-j}(B_{n-j})V_{j}(K) \qquad \forall \varepsilon\geq 0,
\end{equation}
where $B_m$ denotes the $m$-dimensional Euclidean unit ball centered at the origin $o$ and $K+\varepsilon B_n=\{x+\varepsilon y: x\in K, y\in B_n\}$. \emph{Kubota's integral formula} provides an explicit representation of the intrinsic volumes in terms of the volumes of  projections of the body, namely, for $j\in[n]=\{1,\ldots,n\}$,
\begin{equation}\label{kubota}
    V_j(K)=\bbinom{n}{j}
        \int_{\Gr(n,j)}\vol_j(K|H)\,\dint \nu_j(H),
\end{equation}
where $\Gr(n,j)$ is the Grassmannian of all $j$-dimensional subspaces of $\R^n$, $\nu_j$ is the (uniquely
determined) Haar probability measure on $\Gr(n,j)$,  $K|H$ denotes the orthogonal projection of $K$ into the subspace  $H\in\Gr(n,j)$, and
$\tbbinom{n}{j}$ is the \emph{flag coefficient} defined by Klain and Rota in \cite[p. 63]{Klain-Rota}, and appearing in \eqref{ball-binom} below. A celebrated result from integral geometry, Hadwiger's Characterization Theorem \cite{Hadwiger1957}, states that the intrinsic volumes span the space of all continuous rigid motion invariant valuations on  the convex bodies in $\R^n$. They are normalized so that if $L$ is a $j$-dimensional convex body, then $V_j(L)$ is the $j$-dimensional Lebesgue measure of $L$. In particular,  $V_n(K)=\vol_n(K)$ is the $n$-dimensional volume of $K$, $V_{n-1}(K)$ is half the surface area of $K$, $V_1(K)$ is a constant multiple of the mean width of $K$ and $V_0(K)=1$ is the Euler characteristic. For more background on intrinsic volumes and valuations, we refer the reader to Chapters 4 and 6 in the book \cite{SchneiderBook} by Schneider. 

\medskip
A fundamental question is to compare the intrinsic volumes of a pair of convex bodies, and for this purpose, several notions of intrinsic volume ``distance" have been introduced and studied in the literature. 
For convex bodies $K$ and $L$ in $\R^n$ and $j\in[n]$, Florian \cite{Florian-1989} defined (a constant multiple of) the following intrinsic volume ``distance"\!,
\begin{equation*}
	\rho_j(K,L) :=2V_j([K,L])-V_j(K)-V_j(L)
\end{equation*}
where $[K,L]$ denotes the convex hull of $K\cup L$. It was  shown in \cite{Florian-1989} that $\rho_j$ is a metric for $j=1$ but is not a metric for $j\in\{2,\ldots,n\}$; we thus call $\rho_1$ the \emph{mean width metric}.
Later, the \emph{$j$th intrinsic volume deviation} $\Delta_j$ was defined in \cite{BHK}  by
\begin{equation}\label{int-vol-dev-def}
    \Delta_j(K,L):=V_j(K)+V_j(L)-2V_j(K\cap L)
\end{equation}
and was used to study the asymptotic best and random approximations of convex bodies by polytopes. In particular, $\Delta_n$ is the symmetric difference metric and $\Delta_{n-1}$ is half of the surface area deviation (see \cite{Groemer2000,ShephardWebster-1965}). It was shown in \cite[Lem. 24]{BHK} that $\Delta_j$ does not satisfy the triangle inequality if $j\in[n-1]$; hence it is a metric only for $j=n$. 

For each $j\in[n]$, the functionals $\rho_j$ and $\Delta_j$ are positive definite in the sense that if $K$ and $L$ are convex compact sets with $\dim(K)\geq j$ and $\dim(L)\geq j$, then $\rho_j(K,L)=0$ implies $K=L$ and $\Delta_j(K,L)=0$ implies $K=L$. Therefore,  $\rho_j$ and $\Delta_j$ may be considered as measures of deviation between convex bodies, and they have appeared in the literature before in this context (see, for example, \cite{BHK}). To our knowledge, a definition has not yet been put forth for an intrinsic volume metric which satisfies the triangle inequality for every $j\in[n]$. In view of the Kubota formula, we make the following

\begin{definition}\label{mainDefn}
For convex bodies $K$ and $L$ in $\R^n$ and $j\in[n]$, we define the $j$th \emph{intrinsic volume metric} $\delta_j$ by
\begin{equation}\label{def:metric}
    \delta_j(K,L) := \bbinom{n}{j} \int_{\Gr(n,j)} \vol_j((K|H)\triangle (L|H))\, \dint\nu_j(H),
\end{equation}
where $\tbbinom{n}{j}$ is the flag coefficient defined  in  \eqref{ball-binom} and $\triangle$ denotes the symmetric difference of sets.
\end{definition}

This quantity may be thought of as the mean distance of the shadows of $K$ and $L$, averaged over all $j$-dimensional subspaces. Note that $\delta_j(K,L)=\tbbinom{n}{j}\int_{\Gr(n,j)}\Delta_j(K|H,L|H)\,\dint \nu_j(H)$, where $\Delta_j$ is the symmetric difference metric in the $j$-dimensional subspace $H$;  in particular,  $\delta_n(K,L)=\vol_n(K\triangle L)$. In Theorem \ref{metricthm}, we show that $\delta_j$ is a (pseudo-)metric on $\cK^n$ for all $j\in[n]$.

\vspace{2mm}

In the special case $K\subset L$, all three notions of distance reduce to the \emph{$j$th intrinsic volume difference} of $K$ and $L$, 
\begin{equation*}
    \delta_j(K,L)=\rho_j(K,L)=\Delta_j(K,L)=V_j(L)-V_j(K),
\end{equation*}
which has been studied in the best and random approximations of convex bodies by polytopes; see, for example, \cite{Affentranger1991, BFV2010, BHK, BFH2013, BHH2008, glasgrub,Reitzner2002} and the references therein. In the special case of the Euclidean ball, precise estimates for the asymptotic \emph{best} approximation by inscribed polytopes with $N$ vertices (or circumscribed polytopes with $N$ facets) were given in \cite[Thm. 1, i)]{BHK}. There  it was shown that for any $j\in[n]$ there exist positive absolute constants $c_1,c_2$ such that for all sufficiently large $N$, there exists a polytope $P_{n,j,N}\subset B_n$ with at most $N$ vertices (respectively, $P_{n,j,N}\supset B_n$ with at most $N$ facets) which satisfies 
\begin{equation}\label{best-inscribed}
    c_1 jV_j(B_n)\NN \leq \Delta_j(B_n,P_{n,j,N}) \leq c_2 jV_j(B_n)\NN.
\end{equation}
Furthermore, it was shown in \cite[Rmk. 2]{BHK} that $c_1\sim c_2=\frac{1}{2}+O\left(\frac{\ln n}{n}\right)$ as $n\to\infty$.

One may also drop the containment restriction and consider approximation by \emph{arbitrarily positioned polytopes}. This yields an estimate which is at least as good as the best inscribed (or circumscribed) approximation, and it is interesting to ask by how much. Specifically, it was shown in \cite[Thm. 1 iii)]{BHK} that there exists a positive absolute constant $C$ such that for all sufficiently large $N$, there exists an arbitrarily positioned polytope $Q_{n,j,N}$ with at most $N$ vertices such that 
\begin{equation}\label{BHK-upper-bd}
    \Delta_j(B_n,Q_{n,j,N}) \leq C\min\left\{1,\frac{j\ln n}{n}\right\} V_j(B_n)\NN.
\end{equation}
Comparing \eqref{best-inscribed} and \eqref{BHK-upper-bd}, we see that in the special cases of the symmetric difference metric ($j=n$) and surface area deviation ($j=n-1$), dropping the restriction that the ball contains the polytope (or vice versa) improves the estimate by at least a  factor of dimension (see \cite{BHK,  GTW2021, GW2018, HSW, Kur2017, LSW} and Remark \ref{rmk-compare-best} below). The same phenomenon has also been observed for the mean width metric $\rho_1$ \cite{BHK, glasgrub,HK-DCG,Kur2017, Ludwig1999,LSW}.

By  definition of the orthogonal projection, for any convex bodies $K$ and $L$ with $K\cap L\neq\varnothing$ and any $H\in\Gr(n,j)$, it holds that $(K\cup L)|H=(K|H)\cup (L|H)$ and $(K\cap L)|H\subset (K|H)\cap (L|H)$.  Hence, $\delta_j\leq \Delta_j$. From this inequality and the preceding discussion, it is natural to ask:
\begin{quote}
    \emph{How much can we improve the  upper bound \eqref{BHK-upper-bd} for the approximation if
    we measure the distance by $\delta_j$ instead of $\Delta_j$?}
\end{quote}
\vspace{1mm}
\noindent We give an answer in our main result.

\begin{theorem}\label{mainThm}
There exists an absolute constant $C$ such that for every $n\in\mathbb{N}$ with $n\geq 2$ and every $j\in[n]$, there exists an integer $N_{n,j}$ such that for all $N\geq N_{n,j}$, there exists a polytope $Q_{n,j,N}^*$ with at most $N$ vertices which satisfies
\begin{equation}\label{main-result}
    \delta_j(B_n,Q_{n,j,N}^*) \leq C\frac{j}{n}V_j(B_n)\NN.
\end{equation}
\end{theorem}

\begin{remark}
More specifically, under the hypotheses of Theorem \ref{mainThm}, the following estimate for the constant $C$ is achieved:
\begin{align}\label{main-ineq}
C = 2 + O\left(\tfrac{\ln n}{n}\right) \qquad \text{as $n\to\infty$}.
\end{align}
\end{remark}

\begin{remark}
Let $n,N\in\mathbb{N}$ with $n\geq 2$ and $N\geq n+1$. It follows from a compactness argument and the continuity of $\delta_j$ with respect to the Hausdorff metric, see Theorem \ref{metricthm} ii), that for each $j\in[n]$, there exists a \emph{best-approximating polytope} $Q_{n,j,N}^{\mathrm{best}}$ such that
\begin{equation*}
\delta_j(Q_{n,j,N}^{\mathrm{best}},B_n)= \min\{\delta_j(Q_{n,N},B_n): \, Q_{n,N}\text{ is a polytope in }\R^n\text{ with at most }N\text{ vertices}\}.
\end{equation*}
As a minimizer, a best-approximating polytope also satisfies the bound in Theorem \ref{mainThm}. 
\end{remark}

\begin{remark}\label{rmk-compare-best}
Let $j\in[n]$. If $j\leq \frac{n}{\ln n}$, then Theorem \ref{mainThm} improves the bound in \eqref{BHK-upper-bd} by a factor of $\ln n$, and if $j> \frac{n}{\ln n}$, then Theorem \ref{mainThm} improves the bound in \eqref{BHK-upper-bd} by a factor of $\frac{n}{j}$.
Comparing Theorem \ref{mainThm} and \eqref{best-inscribed}, we see that dropping the restriction that the polytope is contained in the ball (or vice versa) improves the estimate by a factor of dimension. The same phenomenon has already been observed  for the symmetric difference metric \cite{GW2018, Gruber93, Kur2017,Ludwig1999, LSW},  surface area deviation \cite{BHK,GTW2021, HSW, Kur2017} and  mean width metric \cite{glasgrub,Gruber93,HK-DCG, Kur2017,Ludwig1999,LSW}.
\end{remark}

\subsection{Background and notation}

 In what follows, let $\dim(K)$ denote the dimension of the affine hull of a convex compact set $K$ in $\R^n$. For $n\in\mathbb{N}$ and $j\in[n]$, we set $\cK_j^n:=\{K\in\cK^n : \dim(K)\geq j\} = \{K\in\cK^n : V_j(K)>0\}$. In particular, $\cK_n^n$ is the class of convex bodies in $\R^n$ (convex, compact subsets $K$ with nonempty interior). Notice that $\cK_n^n\subset \cK_{n-1}^n\subset \cdots\subset \cK_0^n=\cK^n$. 

The volume $\vol_n(B_n)$ and surface area $\omega_n$ of the $n$-dimensional unit ball $B_n$ are given by the formulas
\begin{equation*}
    \vol_n(B_n) = \frac{\pi^{n/2}}{\Gamma(1+n/2)} \qquad \text{and}\qquad 
    \omega_n =n\vol_n(B_n)=\frac{2\pi^{n/2}}{\Gamma(n/2)},
\end{equation*}
where $\Gamma$ is the gamma function. For $j\in[n]$, the \emph{flag coefficient} $\tbbinom{n}{j}$ is defined by
\begin{equation}\label{ball-binom}
    \bbinom{n}{j} := \binom{n}{j} \frac{\vol_n(B_n)}{\vol_{j}(B_j) \vol_{n-j}(B_{n-j})} = \frac{1}{2} \frac{\omega_{j+1}\omega_{n-j+1}}{\omega_{n+1}},
\end{equation}
where the identity follows by the Legendre duplication formula.
The unit sphere in $\R^n$ is denoted by $\S^{n-1}$, and the uniform probability measure on $\S^{n-1}$ is denoted by $\sigma$. Throughout the paper, $C,c,c_1,C_2$, etc., will denote positive absolute constants whose values may change from line to line. To be clear, the dependence of a positive constant on variable quantities will be denoted explicitly; for example,  $C(n,j)$ denotes a positive constant which depends only on $n$ and $j$.

\subsection{Overview}

In Section \ref{properties-sec} we establish  some of the basic properties of $\delta_j$, and we give estimates for $\delta_j$ in terms of the intrinsic volume ``distances" $\rho_j$ and $\Delta_j$. Next, in Section \ref{sec-background-lemmas}, we provide the relevant background from stochastic and integral geometry needed to prove Theorem \ref{mainThm}. The proof of Theorem \ref{mainThm} is given in Section \ref{sec-proof-mainThm}. Finally, estimates for some of the constants appearing in the proofs are carried out in  Appendix \ref{sec-appendix}, and in Appendix \ref{sec:appendix-B} we briefly verify that a result of Affentranger remains true in the $1$-dimensional case (see Lemma \ref{beta-volume}).

\section{Properties of the intrinsic volume metric}\label{properties-sec}

There are numerous geometrically reasonable and interesting ways to define a metric on the class $\mathcal{K}^n$ of convex bodies. Perhaps the most prominent example is the Hausdorff metric, which plays a central role in convex geometry and functional analysis due to its many useful properties (see, for example, \cite[Ch. 1.8]{SchneiderBook}). 
For $K,L\in\mathcal{K}^n$, the \emph{Hausdorff metric} $d_H:\mathcal{K}^n\times\mathcal{K}^n\to[0,\infty)$ is defined by 
		\begin{equation*}
		    d_H(K,L) = \min \{ \lambda\geq 0 : K\subset L+\lambda B_n,\, L\subset K+\lambda B_n\}.
		\end{equation*}
The metric $\delta_j$ also has several  useful properties, which it inherits from the $j$-dimensional Lebesgue measure, including:

\begin{theorem}[Properties of $\delta_j$ on $\cK_j^n$]\label{metricthm}
    Let $j\in[n]$ and let $\delta_j:\cK^n\times \cK^n \to [0,\infty)$ be the functional defined by \eqref{def:metric}. Then $\delta_j$ is:
    \begin{itemize}
    \item[(i)] a metric on $\cK_k^n$ if $k\geq j$, and a pseudometric if $k<j$;
		\item[(ii)] continuous with respect to the Hausdorff metric;
		
		\item[(iii)] rigid motion invariant, that is, 
			\begin{equation*}
				\delta_j(\vartheta K + x,\vartheta L + x) = \delta_j(K,L) 
			\end{equation*}
			for all orthogonal transformations $\vartheta\in \mathrm{O}(\R^n)$ and all translations $x\in \R^n$;
        \item[(iv)] positively $j$-homogeneous, that is, for all $t>0$ we have that
            \begin{equation*}
                \delta_j(tK,tL) = t^j \delta_j(K,L).
            \end{equation*}
	\end{itemize}
\end{theorem}

\begin{proof}
    Note that for $K,L\in\cK^n$, the function $H \mapsto \vol_j((K|H)\cap (L|H))$ is a Borel function on the homogeneous space $\Gr(n,j)$. Thus $\delta_j$ is well-defined on $\cK_j^n$  since $\nu_j$ is a Radon measure (see, for example, \cite[Ch.\ 12 \& Ch.\ 13]{SchneiderWeilBook}).
    The integrand is nonnegative and symmetric. Thus, so too is $\delta_j$. If $K=L$, then the set $(K|H)\triangle(L|H)$ is empty, so $\delta_j(K,L)=0$. If $\delta_j(K,L)=0$, then $\vol_j((K|H)\triangle(L|H))=0$ for almost all $H\in\Gr(n,j)$.
    For $K,L\in \cK_j^n$ we have that $\min\{\dim(K|H),\dim(L|H)\}=j$ for almost all $H\in\Gr(n,j)$.
    Thus, since $\Delta_j$ is a metric on the class of $j$-dimensional convex bodies in $H$, $\delta_j(K,L)=0$ implies that $K|H = L|H$ for almost all $H\in\Gr(n,j)$. Now the proof of \cite[Thm. 3.1.1]{GardnerBook} can be modified to show that if $K|H=L|H$ for almost all $H$, then $K=L$. 
    Furthermore, since the integrand satisfies the triangle inequality, so too does $\delta_j$. 
    Finally, if $K,L\in\cK^n$, $K\neq L$ and $\max\{\dim(K),\dim(L)\} < j$, then $\Delta_j((K|H),(L|H)) = \vol_j((K|H)\triangle (L|H)) = 0$ for all $H\in\Gr(n,j)$ since $\dim((K|H)\triangle (L|H))\leq \max\{\dim(K),\dim(L)\}<j$. This proves (i).

    For part (ii),  consider a sequence $\{K_i\}$ in $\cK^n$ that converges to $K\in\cK^n$ with respect to the Hausdorff metric $d_H$. Then for $H\in\Gr(n,j)$, by the continuity of $L\mapsto L|H$ we find that $K_i|H\to K|H$ with respect to $d_H$.  Therefore, by the dominated convergence theorem, 
    \begin{equation*}
        \lim_{i\to\infty} \delta_j(K_i,K) = \bbinom{n}{j} \int_{\Gr(n,j)} \lim_{i\to\infty} \Delta_j(K_i|H,K|H) \, \dint\nu_j(H) = 0
    \end{equation*}
    since $\Delta_j$ is continuous with respect to $d_H$ (see, for example,  \cite[Thm.\ 7]{ShephardWebster-1965}).
    For sequences $\{K_i\}$ and $\{L_i\}$ that converge to $K$ and $L$, respectively,  we find that by the triangle inequality,
    \begin{equation*}
        \lim_{i\to \infty} |\delta_j(K_i,L_i)-\delta_j(K,L)| \leq \lim_{i\to\infty} \big(\delta_j(K_i,K)+\delta_j(L_i,L)\big) = 0
    \end{equation*}
    with respect to the Hausdorff metric. 
    Thus $\delta_j$ is continuous with respect to the Hausdorff metric.

    Parts (iii) and (iv) follow since the functional $\vol_j$ in the integrand is rigid motion invariant and positively $j$-homogeneous, respectively, and by the invariance of $\nu_j$ with respect to $\mathrm{O}(\R^n)$.
\end{proof}

The Hausdorff metric space $(\cK^n,d_H)$ is complete, and so is the symmetric difference metric space $(\cK_n^n\cup\{\varnothing\},\Delta_n)$. In the latter case, we need to append the empty set to account for sequences of convex bodies $\{K_i\}\subset\cK^n_n$ that converge to a lower-dimensional convex set with respect to the Hausdorff metric, that is, if $K_i\to K_0$ and $\dim(K_0) < n$, then $\Delta_n(K_i,\varnothing) = \vol_n(K_i) \to 0$ and therefore $K_i\to \varnothing$ with respect to $\Delta_n$.
Furthermore, $d_H$ and $\Delta_n=\delta_n$ are equivalent on $\cK_n^n$. For more background, see the article \cite{ShephardWebster-1965} by Shephard and Webster.

We also conjecture that  $d_H$ and $\delta_j$ induce the same topology on $\cK_j^n$. We therefore pose the following 
\begin{question}[uniform bound]
    Let $j\in\{2,\dotsc,n-1\}$. If $\{K_i\}$ is a sequence in $\cK_j^n$ that converges to $K\in\cK_j^n$ with respect to $\delta_j$, does it follow that $\{K_i\}$ is uniformly bounded, that is, does there exist $R>0$ such that $\bigcup_{i\in\mathbb{N}} K_i \subset RB_n$?
\end{question}
A positive answer to this question would immediately imply (by the Blaschke Selection Theorem) that a sequence $\{K_i\}$ which converges with respect to $\delta_j$ also converges with respect to $d_H$.
For $j=n$, this question was answered in the positive in \cite[Lem.\ 11]{ShephardWebster-1965}, and for $j=1$ see \cite[Thm.\ 3]{Vitale-1985} and \cite[Thm.\ 2]{Florian-1989}.

\begin{remark}
Similiar to \cite[Thm.\ 13]{ShephardWebster-1965}, we find that the identity mapping $\mathrm{id}\colon (\cK^n,d_H) \to (\cK^n,\delta_j)$ is not uniformly continuous for $j\geq 2$.  For example, if $K_\ell = \ell B_n$ and $L_\ell = (\ell+\frac{1}{\ell})B_n$, then $d_H(K_\ell,L_\ell) = \frac{1}{\ell} \to 0$ as $\ell\to \infty$, but
\begin{equation*}
    \delta_j(K_\ell,L_\ell) = V_j(L_\ell)-V_j(K_\ell) = \left[\left(\ell+\frac{1}{\ell}\right)^j-\ell^j\right] V_j(B_n) \not\to 0 \quad \text{as $\ell\to \infty$.}
\end{equation*}
For $j=1$, the identity mapping is uniformly continuous since
\begin{equation*}
    \delta_1(K,L) \leq \bbinom{n}{1} d_H(K,L).
\end{equation*}
\end{remark}

\subsection{Independence of \texorpdfstring{$\delta_j$}{delta j} with respect to the ambient space}

The Hausdorff metric is intrinsic in the following sense: If $K,L\subset \R^{n+\ell}$, $\ell>0$, are $n$-dimensional convex bodies contained in $\R^n\subset \R^{n+\ell}$, then $d_H^{n+\ell}(K,L) = d_H^n(K,L)$, where $d_H^m$ denotes the Hausdorff metric defined in $\R^m$. Thus the Hausdorff metric does not depend on the dimension of the ambient space and can be evaluated in the affine hull of $K\cup L$.
Another point of view is that $(\cK^n,d_H^n)$ can be isometrically embedded in $(\cK^{n+\ell},d_H^{n+\ell})$. Next, we show that the same independence of the dimension also holds for the intrinsic volume metric. In what follows, let $\delta_j^m$ denote the $j$th intrinsic volume metric in $\R^m$.

\begin{theorem}[$\delta_j$ is independent of the ambient space dimension]\label{thm:intrinsic}
    Let $n,\ell\in\N$ and $j\in[n]$. 
	Then the restriction of $(\cK_j(\R^{n+\ell}),\delta_j^{n+\ell})$ to $\cK_j(\R^n)$ is isometric to $(\cK_j(\R^n), \delta_j^n)$.
\end{theorem}

To prove this theorem, we first recall some facts about projecting between linear subspaces. 
Given two linear $j$-dimensional subspaces $U$ and $V$ in $\R^n$, we define the (asymmetric) angle $\Theta_{U,V}\in[0,\frac{\pi}{2}]$ by
\begin{equation*}
    \cos\Theta_{U,V} = \frac{\vol_j(B_j^U|V)}{\vol_j(B_j)},
\end{equation*}
where $B_j^U$ denotes the $j$-dimensional Euclidean unit ball in $U$.
We call $\cos\Theta_{U,V}$ the \emph{projection factor} of $U$ on $V$ and in the literature one also often finds the notation $|\langle U, V\rangle| = \cos\Theta_{U,V}$.
Note that $\Theta_{U,V} = \Theta_{V^\bot,U^\bot} = \Theta_{U,U|V}$, and if $\dim U = \dim V$ then $\Theta_{U,V}=\Theta_{V,U}$. Furthermore, we have that
\begin{equation*}
    \cos\Theta_{U,V} = \prod_{i=1}^j \cos \vartheta_i,
\end{equation*}
where $0\leq \vartheta_1 \leq \cdots \leq \vartheta_j \leq \frac{\pi}{2}$ are the principal angles between $U$ and $V$ (also called canonical or Jordan angles). We refer to \cite{Mandolesi-2021} for further references. Chern \cite[eq.\ (24)]{Chern-1966} showed that
\begin{equation}
    c_{n,j,\ell} = \int_{\Gr(n,j)} \left(\cos\Theta_{H,\R^j}\right)^\ell \, \dint \nu_j(H) 
    = \frac{\overline{\omega}_{n+\ell}\, \overline{\omega}_{\ell}\, \overline{\omega}_{j}\, \overline{\omega}_{n-j}}{\overline{\omega}_n\,\overline{\omega}_{l+j}\,\overline{\omega}_{n-j+\ell}}
\end{equation}
for $\ell\in\N$, $j\in[n]$ and $\overline{\omega}_p = \prod_{k=1}^p \omega_k$. The following integral-geometric formula follows from a related formula established by Chern \cite[eq.\ (28)]{Chern-1966} (see also Santal\'o \cite[eq.\ (14.40)]{Santalo-2004}). We write $\Gr(V,p)$ for the Grassmannian of all $p$-dimensional linear subspaces in a linear space $V$ and $\nu_p^V$ for the corresponding invariant Haar probability measure on $\Gr(V,p)$.
\begin{theorem}[\cite{Chern-1966}]\label{thm:chern}
Let $n\in\N$, $p,q\in[n]$ and $q\leq p$. Fix $V\in\Gr(n,n-q)$. If $f:\Gr(n,n-p)\to [0,\infty)$ is integrable, then
\begin{align*}
    &c_{n-p+q,q,p-q} \int_{\Gr(n,n-p)} f(H)\, \dint \nu_{n-p}(H) \\
    &\qquad =  \int_{\Gr(V,n-p)} \int_{\Gr(H'+V^\bot,n-p)} \left(\cos\Theta_{H,H'}\right)^{p-q} f(H) \, \dint \nu_{n-p}^{H'+V^\bot}(H) \,\dint \nu_{n-p}^V(H').
\end{align*}
\end{theorem}
\begin{proof}
    The integral-geometric formula \cite[eq.\ (28)]{Chern-1966} of Chern  shows that
    \begin{equation*}
        c_{n-p+q,q,p-q} \int_{\Gr(n,p)} f(L^\bot) \, \dint\nu_p(L) = \int_{\Gr(V,p-q)} \int_{\Gr^{[E']}(n,p)} \!\!\!\!\!\left(\cos \Theta_{L^\bot, V}\right)^{p-q} f(L^\bot) \, \dint \nu_p^{[E']}(L) \, \dint \nu_{p-q}^V(E'),
    \end{equation*}
    where $\Gr^{[E']}(n,p):=\{ L \in \Gr(n,p) : L\supset E'\}$ and $\nu_p^{[E']}$ is the invariant Haar probability measure on $\Gr^{[E']}(n,p)$.
    Note that the constant $c_{n-p+q,q,p-q}$ is related to the constant that appears in \cite[p.\ 106]{Chern-1966}, where the difference is due to the fact that we are considering the probability measure $\nu_p$ on $\Gr(n,p)$ instead of the Hausdorff measure. Furthermore, in \cite[eq.\ (28)]{Chern-1966} the function $D=D(L,V)>0$ is used instead of $\cos\Theta_{L^\bot,V}$. To see that this change is allowed, we note that Chern assumed that $\dim(L\cap V) = \dim(E^p\cap (E_0^q)^\bot)=p-q$, which is true for almost all $L=E^p\in\Gr(n,p)$ for a fixed $V=(E_0^q)^\bot\in \Gr(n,n-q)$, and considered two orthogonal bases $e_1,\dotsc,e_n$ and $f_1,\dotsc,f_n$ of $\R^n$ such that:
    \begin{itemize}
        \item $e_i=f_i$ for $1\leq i \leq p-q$ and $e_1,\dotsc,e_{p-q}$ spans $L\cap V = E^p\cap(E_0^q)^\bot$,
        \item $e_1,\dotsc,e_p$ spans the $p$-dimensional linear subspace $L=E^p$,
        \item $e_{p+1},\dotsc,e_n$ spans the $(n-p)$-dimensional linear subspace $L^\bot= (E^p)^\bot$, 
        \item $f_1,\dotsc,f_{n-q}$ spans the $(n-q)$-dimensional linear subspace $V = (E_0^q)^\bot$, and
        \item $f_{n-q+1},\dotsc,f_n$ spans the $q$-dimensional linear subspace $V^\bot = E_0^q$.
    \end{itemize}
    Under these conditions, there is a proper orthonormal transformation $U\in \mathrm{SO}(n)$ such that $e_i = \sum_{j=1}^n u_{ij} f_j$ for all $1\leq i \leq n$.
    Now the orthogonal projection $\cdot|V$ from $\R^n$ to $V$ acts on $e_\lambda$ by
    \begin{equation*}
        e_\lambda|V = \sum_{s=n-p+1}^n u_{\lambda s}\, f_s|V = \sum_{s=p-q+1}^{n-q} u_{\lambda s}\, f_s \quad \text{for all $p+1\leq \lambda \leq n$},
    \end{equation*}
    where we used the fact that $e_1,\dotsc,e_{p-q},f_{p-q+1},\dotsc,f_{n-q}$ spans $V=(E_0^q)^\bot$, and therefore $f_s|V=f_s$ for $p-q+1\leq s \leq n-q$ and $f_s|V=0$ for $s\geq n-q+1$.
    This implies that:
    \begin{itemize}
        \item $f_{p-q+1},\dotsc,f_{n-q}$ spans the $(n-p)$-dimensional linear subspace $L^\bot|V$, and 
        \item $e_{p-q+1},\dotsc,e_p$ spans the $q$-dimensional linear subspace $V^\bot|L$.
    \end{itemize}
    Chern considered the submatrix\footnote{There is apparently a misprint in \cite[eq.\ (27)]{Chern-1966}, where it is stated that $D=\det(u_{\lambda \mu})$ instead of $D=\det(u_{\lambda s})$. 
    Indeed, using the notation of the paper, we have that $(de_\alpha, e_\lambda) = \sum_{s=p-q+1}^{n-q} u_{\lambda s} (de_{\alpha}, f_s)$.}
    $U'=\left(u_{\lambda s}\right)_{\substack{p+1\leq \lambda\leq n\\ p-q+1\leq s\leq n-q}}$ which determines the relative position between the $(n-p)$-dimensional spaces $L^\bot$ and $L^\bot|V$.
    Thus
    $D=\det(U')=\det(u_{\lambda s})$ is equal to the projection factor $\cos\Theta_{L^\bot,V}=\cos\Theta_{L^\bot, L^\bot|V}$ of $L^\bot=(E^p)^\bot$ to $V=(E_0^q)^\bot$.
    
    Finally, note that for $E'\in\Gr(V,p-q)$,
    \begin{align*}
        \Gr^{[E']}(n,p) 
            &= \{ L \in \Gr(n,p) : L\supset E'\} \\
            &= \{ H^\bot :  H\in \Gr(n,n-p), H\subset (E')^{\bot_V} + V^\bot\} \\
            &= \{H^\bot : H\in \Gr((E')^{\bot_V}+V^\bot, n-p)\},
    \end{align*}
    where we write $\bot_V$ for the orthogonal complement of a subspace in $V$.
    The statement of the theorem then follows by the change of variables $H' = (E')^{\bot_V}$ 
 (respectively, $H = L^\bot$), where we note that the push-forward of $\nu_{p-q}^V$ under the map $E'\mapsto (E')^{\bot_V}$ from $\Gr(V,p-q)$ to $\Gr(V,n-p)$ is $\nu_{n-p}^V$ (respectively, the push-forward of $\nu_p$ and under the map $L\mapsto L^\bot$ from $\Gr(n,p)$ to $\Gr(n,n-p)$ is the measure $\nu_{n-p}$), and the fact that $\cos\Theta_{H,V} = \cos\Theta_{H,H|V} = \cos\Theta_{H,H'}$. Also note that the push-forward of $\nu_p^{[(H')^{\bot_V}]}$ under the map $H\mapsto H^\bot$ from $\Gr^{[(H')^{\bot_V}]}(n,p)$ to $\Gr(H'+V^\bot,n-p)$ is the measure $\nu_{n-p}^{H'+V^\bot}$.
\end{proof}

If we apply Theorem \ref{thm:chern} to the function $H\mapsto \vol_j(K|H)$ for $H\in\Gr(n+\ell,j)$ and a fixed convex body $K\subset V=\R^n$, where we identify $(\R^n)^\bot\subset \R^{n+\ell}$ with $\R^\ell$, then
\begin{align*}
    &c_{j+\ell,\ell,n-j} \int_{\Gr(n+\ell,j)} \vol_j(K|H) \, \dint \nu_j(H) \\
        &\qquad = \int_{\Gr(n,j)} \int_{\Gr(H'+\R^\ell, j)} \left(\cos \Theta_{H,H'}\right)^{n-j} \vol_j(K|H) 
            \, \dint \nu_j^{H'+\R^{\ell}}(H) \, \dint\nu_{j}^{\R^n}(H')\\
        &\qquad = \int_{\Gr(n,j)} \vol_j(K|H') \int_{\Gr(H'+\R^\ell, j)} \left(\cos \Theta_{H,H'}\right)^{n-j+1} 
            \, \dint \nu_j^{H'+\R^{\ell}}(H) \, \dint\nu_{j}^{\R^n}(H')\\
        &\qquad = c_{j+\ell,\ell,n-j+1} \int_{\Gr(n,j)} \vol_j(K|H') \,\dint\nu_{j}(H'),
\end{align*}
where we used the fact that $\vol_j(K|H) = (\cos\Theta_{H,H'}) \vol_j(K|H')$, which holds due to the special positions of $K,H$ and $H'$. We have that
\begin{equation*}
    \frac{c_{j+\ell,\ell,n-j+1}}{c_{j+\ell,\ell,n-j}} 
    = \frac{\omega_{n+\ell+1}}{\omega_{n+1}} \frac{\omega_{n-j+1}}{\omega_{n-j+\ell+1}} 
    = \bbinom{n}{j} \bbinom{n+\ell}{j}^{-1},
\end{equation*}
which yields
\begin{equation*}
    V_j(K) = \bbinom{n+\ell}{j} \int_{\Gr(n+\ell,j)} \vol_j(K|H) \, \dint \nu_j(H) = \bbinom{n}{j} \int_{\Gr(n,j)} \vol_j(K|H') \, \dint \nu_j(H').
\end{equation*}
This again verifies that $V_j(K)$ is  independent of the dimension of the ambient space; see also \cite[Thm.\ 6.2.2]{SchneiderWeilBook}.

We are now ready to prove that $\delta_j(K,L)$ is also independent of the dimension of the ambient space.
\begin{proof}[Proof of Theorem \ref{thm:intrinsic}]
    Let $K,L\in \cK_j(\R^n) \subset \cK_j(\R^{n+\ell})$. We set $V=\R^n$ and identify $V^\bot=\R^\ell\subset \R^{n+\ell}$.
    We apply Theorem \ref{thm:chern} to the measurable functions $H\mapsto \vol_j(K|H)$, $H\mapsto \vol_j(L|H)$ and $H\mapsto \vol_j((K|H)\cap (L|H))$,
    and we claim that  for $H'\in \Gr(V,j)$,
    \begin{equation}\label{eqn:claim}
        \vol_j((K|H)\cap (L|H)) = \vol_j([(K|H')\cap (L|H')]|H) \quad \text{for almost all $H\in\Gr(H'+V^\bot,j)$}.
    \end{equation}
    Since $\vol_j(\cdot|H) = (\cos\Theta_{H,H'}) \vol_j(\cdot|H')$, Theorem \ref{thm:intrinsic} will follow since by \eqref{eqn:claim}, 
    \begin{align*}
        \delta_j^{n+\ell}(K, L) 
        &= \bbinom{n+\ell}{j} \int_{\Gr(n+\ell,j)} \left[\vol_j(K|H) + \vol_j(L|H) - 2 \vol_j((K|H)\cap (L|H)) \right]\, \dint \nu_j(H)\\
        &= c_{j+\ell,\ell,n-j}^{-1} \bbinom{n+\ell}{j} \int_{\Gr(V,j)} \bigg[\vol_j(K|H')+\vol_j(L|H')-2\vol_j((K|H')\cap (L|H'))\\
        &\qquad \qquad \qquad \quad    \times \int_{\Gr(H'+V^\bot,j)} \left(\cos\Theta_{H,H'}\right)^{n-j+1}\, \dint \nu_j^{H'+V^\bot}(H) \bigg] \dint \nu_j^V(H')\\
        &= \frac{c_{j+\ell,\ell,n-j+1}}{c_{j+\ell,\ell,n-j}} \bbinom{n+\ell}{j} \bbinom{n}{j}^{-1} \delta_j^n(K, L) = \delta_j^n(K,L).
    \end{align*}
    To prove claim \eqref{eqn:claim}, we note that $K|H = (K|H')|H$ and $L|H= (L|H')|H$ since $H' = H|V$ and $K,L\subset \R^n=V$.
    Hence,
    \begin{equation*}
        (K|H) \cap (L|H) = [(K|H')|H] \cap [(L|H')|H] \subset [(K|H')\cap (L|H')]|H.
    \end{equation*}
    To prove the other inclusion we may assume that $H\cap (H')^\bot = \{o\}$, which holds true for almost all $H\in\Gr(H'+V^\bot,j)$ with respect to $\nu_j$. Let $z\in [(K|H')\cap (L|H')]|H$ be arbitrary. There exist $x\in K$ and $y\in L$ such that
     \begin{equation*}
         z = (x|H')|H = (y|H')|H = x|H = y|H.
     \end{equation*}
     Since $H\cap (H')^\bot=\{o\}$, the orthogonal projection from $H$ to $H'$ is a bijection. It follows that $x|H'=y|H'$, and therefore $x|H'\in(K|H')\cap (L|H')$. This implies that  $z= x|H = (x|H')|H \in [(K|H')\cap (L|H')]|H$. Thus the claim \eqref{eqn:claim} holds true.
\end{proof}

\subsection{Relation to  other intrinsic volume deviations}

The next result provides inequalities between the intrinsic volume metric $\delta_j$ and the intrinsic volume ``distances" $\rho_j$ and $\Delta_j$.

\begin{proposition}\label{IntVolCompThm}
For all convex bodies $K,L\in\cK_n^n$, the following statements hold:
\begin{itemize}
\item[(i)] If $K\subset L$ and $j\in[n]$, then $\delta_j(K,L)=\rho_j(K,L)=\Delta_j(K,L)=V_j(L)-V_j(K)$;

\item[(ii)] $\delta_j(K,L) \leq \min\{\rho_j(K,L),\Delta_j(K,L)\}$ for all $j\in[n]$;

\item[(iii)] $\delta_j(K,L) \geq |V_j(K)-V_j(L)|$ for all $j\in[n]$;

\item[(iv)] $\rho_n(K,L)\geq \delta_n(K,L)=\Delta_n(K,L)=\vol_n(K\triangle L)$;

\item[(v)] If $K\cap L\neq\varnothing$, then 
\begin{equation*}
    \Delta_1(K,L)\geq \delta_1(K,L)=\rho_1(K,L) = 2 \bbinom{n}{1} \int_{\mathbb{S}^{n-1}}|h_K(u)-h_L(u)|\,\dint\sigma(u).
\end{equation*}
\end{itemize}
\end{proposition}

\begin{proof}
Part (i) follows immediately from the definitions. By (i),
\begin{alignat*}{2}
	\Delta_j(K,L) &= \delta_j(K,K\cap L) &&+ \delta_j(L,K\cap L)\\
	\rho_j(K,L) &= \delta_j(K,[K, L]) &&+ \delta_j(L,[K, L]).
\end{alignat*}
Since $\delta_j$ is a metric, statement (ii) now  follows from the triangle inequality. For (iii), 
note that by the triangle inequality,
\begin{align*}
|V_j(K)-V_j(L)|&\leq \left|V_j(K)-\bbinom{n}{j}\int_{\Gr(n,j)}\vol_j((K|H)\cap(L|H))\,\dint\nu_j(H)\right|\\
&+\left|V_j(L)-\bbinom{n}{j}\int_{\Gr(n,j)}\vol_j((K|H)\cap(L|H))\,\dint\nu_j(H)\right|=\delta_j(K,L).
\end{align*}
To prove (iv), note that for $j=n$,
\begin{align*}
\delta_n(K,L)=\Delta_n(K,L)&=\vol_n(K\triangle L)\\
&=2\vol_n(K\cup L)-\vol_n(K)-\vol_n(L)\\
&\leq 2\vol_n([K, L])-\vol_n(K)-\vol_n(L)\\
&=\rho_n(K,L)
\end{align*}
with equality if and only $K\cup L=[K, L]$.

To show (v), let  $\ell_u$ denote the line spanned by $u\in\S^{n-1}$, and let $h_K(u)=\max_{x\in K}\langle x,u\rangle$ be the support function of $K$ in the direction $u$. Then  
$\vol_1(K|\ell_u) = h_K(u)+h_K(-u)$, so
\begin{equation*}
    \vol_1((K|\ell_u)\triangle (L|\ell_u)) = \begin{cases}
\left|h_K(u)-h_L(u)\right| + \left|h_K(-u)-h_L(-u)\right| &\text{if $(K|\ell_u)\cap (L|\ell_u) \neq \varnothing$},\\       h_K(u)+h_K(-u) + h_L(u)+h_L(-u) &\text{if $(K|\ell_u) \cap(L|\ell_u) = \varnothing$}.
\end{cases}
\end{equation*}
Moreover, if $K\cap L\neq \varnothing$, then 
\begin{equation*}
    2\bbinom{n}{1}\int_{\mathbb{S}^{n-1}}|h_K(u)-h_L(u)|\,\dint\sigma(u)= \bbinom{n}{1} \int_{\S^{n-1}} \vol_1((K|\ell_u)\triangle (L|\ell_u)) \, \dint\sigma(u)=\delta_1(K,L).
\end{equation*}
Furthermore, by $(K|\ell_u)\cup(L|\ell_u)=[K,L]|\ell_u$ 
and the  Kubota formula,
\begin{align*}
    \delta_1(K,L)
    &=\bbinom{n}{1}\int_{\mathbb{S}^{n-1}}\big(2\vol_1((K|\ell_u)\cup(L|\ell_u))-\vol_1(K|\ell_u)-\vol_1(L|\ell_u)\big)\,\dint\sigma(u)\\
    &=2V_1([K, L])-V_1(K)-V_1(L)=\rho_1(K,L).
\end{align*}
Finally, the inequality $\Delta_1(K,L)\geq 2\tbbinom{n}{1}\int_{\mathbb{S}^{n-1}}|h_K(u)-h_L(u)|\,\dint\sigma(u)$ 
was shown in  \cite[Thm.\ C1]{BHK}.
\end{proof}

\begin{remark}
In view of Lutwak's \cite{Lutwak1975} dual volumes  $\widetilde{V}_1(K),\ldots,\widetilde{V}_n(K)$ of a convex body $K$ in $\R^n$ and the dual Kubota formula (see, for example, \cite[Sec. 9.3]{SchneiderBook}), for $j\in[n]$ one may also define the \emph{dual volume metric} $\widetilde{\delta}_j:\cK^n\times\cK^n\to[0,\infty)$ by
\[
\widetilde{\delta}_j(K,L):=\bbinom{n}{j}\int_{\Gr(n,j)}\vol_j((K\cap H)\triangle(L\cap H))\,\dint\nu_j(H), \quad K,L\in\cK^n.
\]
We leave it as an exercise to check that $\widetilde{\delta}_j$ is a metric for every $j\in[n]$. In the special case $L=B_n$, it coincides with the $j$th dual volume deviation $\widetilde{\Delta}_j$ defined in \cite{BHK}:
\[
\widetilde{\delta}_j(K,B_n)=\widetilde{V}_j(K)+\widetilde{V}_j(B_n)-2\widetilde{V}_j(K\cap B_n)=:\widetilde{\Delta}_j(K,B_n).
\]
For a $C^2$ convex body $K$ that contains the origin in its interior, limit formulas were established in \cite[Thm. 7]{BHK} for the asymptotic best approximations of $K$ by inscribed, circumscribed and arbitrarily positioned polytopes with respect to $\widetilde{\Delta}_j$.
\end{remark}


\section{Background and lemmas}\label{sec-background-lemmas}

Theorem \ref{mainThm} will be proved using a probabilistic argument, with tools from stochastic and integral geometry. For background on these areas, we refer the reader to  the book \cite{Klain-Rota} by Klain and Rota and the book  \cite{SchneiderWeilBook} by Schneider and Weil, for example.

\subsection{Affine Blaschke--Petkantschin formula}

This result, which is one of the main ingredients of the proof of Theorem \ref{mainThm}, can be found in \cite[Thm. 4]{Miles1971} (see also \cite[Thm.\ 7.2.7]{SchneiderWeilBook}). Recall in the following that $\sigma$ is the uniform probability measure on $\S^{n-1}$.

\begin{lemma}\label{BPext}
For $n\in\mathbb{N}$ with $n\geq 2$, let $g:(B_n)^n\to\mathbb{R}$ be a nonnegative measurable function. Then
\begin{align*}
    &\int_{B_n}\cdots \int_{B_n} g(x_1,\ldots,x_n)\, \dint x_1\ldots \dint x_n \\
    &\quad= n!\vol_n(B_n) \int_{\S^{n-1}}\int_0^1\int_{B_n\cap H}\!\!\!\cdots\int_{B_n\cap H}\!\!\!\!\!g(x_1,\ldots,x_n)\vol_{n-1}([x_1,\ldots,x_n])\,\dint x_1\ldots \dint x_n\,\dint h\,\dint \sigma(u),
\end{align*}
where $H=H(u,h)=\{x\in\mathbb{R}^n:\langle x,u\rangle=h\}$ is the hyperplane orthogonal to $u\in\S^{n-1}$ at distance $h$ from the origin and $[x_1,\dotsc,x_n]$ is the convex hull of the points $x_1,\dotsc,x_n$.
\end{lemma}

\subsection{The random beta polytope model}

Let $X_1,\ldots,X_N$ be independent and identically distributed (i.i.d.) points chosen from $\R^n$ according to the beta distribution, which for a parameter $\beta>-1$ has the density
\begin{equation}\label{beta-density}
f_{n,\beta}(x)=c_{n,\beta}(1-\|x\|^2)^\beta \mathbbm{1}_{\{x: \|x\|<1\}}(x),
\end{equation}
where $c_{n,\beta}$ is a normalization constant given by
\begin{equation}\label{cnb-constant}
c_{n,\beta}=\frac{\Gamma\left(\frac{n}{2}+\beta+1\right)}{\pi^{\frac{n}{2}}\Gamma(\beta+1)}.
\end{equation}
A \emph{random beta polytope} $P_{n,N}^\beta$ is the convex hull of the $X_i$, which is denoted by $[X_1,\ldots,X_N]$. In particular, the uniform probability distribution on the unit ball $B_n$ is given by the $\beta=0$ distribution. Moreover, it was shown by Kabluchko, Temesvari and Th\"ale \cite[Proof of Cor.\ 3.9]{KabluchkoEtAl2019} that the uniform probability measure $\sigma$ on the sphere $\mathbb{S}^{n-1}$ is the weak limit of the beta distribution as $\beta\to -1^+$. 

\subsection{Projections of beta densities onto subspaces}

The next result, which can be found in \cite[Sec.\ 4.2]{KabluchkoEtAl2019}, states that the orthogonal projection of a beta distribution onto a linear subspace yields another beta distribution, but with different parameters. Since the beta distribution is rotationally invariant, without loss of generality the result will be stated for the $j$-dimensional subspace
$L_0:=\{x\in\R^n:\, x_{j+1}=\cdots=x_n=0\}$.

\begin{lemma}[{\cite[Lem.\ 4.4 a)]{KabluchkoEtAl2019}}]\label{KTT-proj1}
If the random variable $X$ has density $f_{n,\beta}$ in $\R^n$, then the orthogonal projection $X|L_0$ has density $f_{j,\beta+\frac{n-j}{2}}$ in $L_0\cong \R^j$.
\end{lemma}

\subsection{Moments of random beta simplices}

 The next result is due to Miles \cite[Equation (72)]{Miles1971} (see also    \cite[Proposition 2.8(a)]{KabluchkoEtAl2019}) and  gives the second moment of the volume of a random beta polytope. 

\begin{lemma}\label{beta-Moments}
Let $\mathcal{V}_{n,n-1}$ denote the $(n-1)$-dimensional volume of the $(n-1)$-dimensional simplex with vertices $X_1,\ldots,X_{n}$ chosen independently  from $B_{n-1}$ according to the beta distribution, that is,  $\mathcal{V}_{n,n-1}:=\vol_{n-1}(P_{n-1,n}^\beta)$. Then the second moment of $\mathcal{V}_{n,n-1}$, denoted $\E[\mathcal{V}_{n,n-1}^2]$, is given by
\begin{align*}
    \E[\mathcal{V}_{n,n-1}^2] 
    =\frac{n}{(n-1)!(n+2\beta+1)^{n-1}}.
\end{align*}
\end{lemma}

We will also need the following formula for the second moment of the weighted volume of a random beta simplex chosen according to the density $f_{n,\beta}$ restricted to a hyperplane. This result is a special case of a result of Kabluchko, Temesvari and Thäle \cite[Lem.\ 4.6]{KabluchkoEtAl2019}, which gives a formula for all nonnegative moments. Here and in the following we set $\dint\P_\beta^H(x)=f_{n,\beta}(x)\, \dint x$ for all $x\in H$, that is, $\P_\beta^H$ is the measure on $H$ that is obtained by the restriction of the beta measure to the affine hyperplane $H$.
\begin{lemma}[{\cite[Lem.\ 4.6]{KabluchkoEtAl2019}}] \label{2nd-moment}
Let $H$ be an $(n-1)$-dimensional hyperplane in $\R^n$ with distance $h$ from the origin. Then for all $h\in[0,1]$,
\begin{align*}
    \int_{B_n\cap H}\!\!\!\! \!\cdots\int_{B_n\cap H} \!\!\!\!\!\! \vol_{n-1}([x_1, \ldots, x_n])^2\, \dint\P_\beta^H(x_1) \cdots \dint \P_\beta^H(x_n)
    =\frac{c_{n,\beta}^n}{c_{n-1,\beta}^n}(1-h^2)^{n\beta+\frac{(n-1)(n+2)}{2}}\E[\mathcal{V}_{n,n-1}^2].
\end{align*}
\end{lemma}
Note that by the Legendre duplication formula,
\[
\frac{c_{n,\beta}}{c_{n-1,\beta}}=2^{-(n+2\beta)}\, \Gamma(n+2\beta+1)\, \Gamma\!\left(\frac{n+2\beta+1}{2}\right)^{\!-2}.
\]
\subsection{Results for the beta measure of a cap of the Euclidean ball}

We follow the definitions and notation in \cite[Sec.\ 4.3]{KabluchkoEtAl2019}. Let $A$ be a measurable subset of $\R^n$. The \emph{probability content} $\mathbb{P}_\beta(A)$ of $A$ is the probability that a random vector with beta distribution attains a value in $A$. For $h\in\R$, define the hyperplane $E_h:=\{x\in\R^n:\, x_n=h\}$, and its two closed halfspaces
\[
E_h^p:=\{x\in\R^n:\,x_n\geq h\}\quad\text{and}\quad E_h^m:=\{x\in\R^n:\,x_n\leq h\}.
\]
Then
\[
\mathbb{P}_\beta(E_h^p)=\int_h^1\int_{E_t}f_{n,\beta}(x)\, \dint x\, \dint t \quad\text{and}\quad \mathbb{P}_\beta(E_h^m)=\int_{-1}^h\int_{E_t}f_{n,\beta}(x)\, \dint x\, \dint t.
\]
The next lemma, also due to Kabluchko, Temesvari and Th\"ale \cite[Lem.\ 4.5]{KabluchkoEtAl2019}, provides a formula for these probability contents in terms of the cumulative distribution function $F_{1,\beta}$ of the one-dimensional beta distribution,
\[
F_{1,\beta}(h) = c_{1,\beta}\int_{-1}^h(1-x^2)^\beta\, \dint x.
\]

\begin{lemma}[{\cite[Lem.\ 4.5]{KabluchkoEtAl2019}}] \label{PC-caps}
Let $h\in\R$ and consider the affine hyperplane $E_h=\{x\in\R^n:\,x_n=h\}$. For the beta distribution on $B_n$ with parameter $\beta>-1$, we have
\[
\mathbb{P}_\beta(E_h^p)=1-F_{1,\beta+\frac{n-1}{2}}(h)\quad\text{and}\quad \mathbb{P}_\beta(E_h^m)=F_{1,\beta+\frac{n-1}{2}}(h), \quad h\in[-1,1].
\]
\end{lemma}

For an affine hyperplane $H=H(h)$ in $\R^n$ with distance $h\in(0,1)$ from the origin $o$, let $H^+$ be the halfspace of $H$ that contains $o$ and let $H^-$ be the halfspace that does not contain $o$. A \emph{cap} of $B_n$ is the intersection of $B_n$ with $H^-$. Note that by the rotational invariance of the beta distribution, for any $h\in(0,1)$ we have $\mathbb{P}_\beta(E_h^p)=\mathbb{P}_\beta(B_n\cap H^-)$. The base of the cap $B_n\cap H^-$ is the $(n-1)$-dimensional Euclidean ball $B_n\cap H$ with radius $r=\sqrt{1-h^2}$. In the next result, we combine \cite[Lem. 1.5]{SW2003} with \cite[Lem. 4.5]{KabluchkoEtAl2019} to derive the beta measure of a cap with respect to its base height.

\begin{lemma}\label{caps-lem}
Let  $\beta\geq -1$. For an affine hyperplane $H=H(h)$  in $\R^n$ with distance $h$ from the origin, let $v=\mathbb{P}_\beta(B_n\cap H^-)=\int_{B_n\cap H^-}f_{n,\beta}(x)\,\dint x$ be the beta measure of the cap of $B_n$ cut off by $H^-$. Then 
\[
\frac{dv}{dh}=-c_{1,\beta+\frac{n-1}{2}}(1-h^2)^{\beta+\frac{n-1}{2}}.
\]
\end{lemma}

\begin{proof}
The case $\beta=-1$ can be found, for example, in \cite[Lem.\ 1.5]{SW2003}. Let $\beta>-1$. By Lemma \ref{PC-caps} and the rotational invariance of the beta distribution, 
\[
v=
\mathbb{P}_\beta(B_n\cap H^-)=1-c_{1,\beta+\frac{n-1}{2}}\int_{-1}^h (1-x^2)^{\beta+\frac{n-1}{2}}\, \dint x.
\]
Hence, by the Fundamental Theorem of Calculus,
\[
\frac{dv}{dh}=-c_{1,\beta+\frac{n-1}{2}}(1-h^2)^{\beta+\frac{n-1}{2}}. \qedhere
\]
\end{proof}

We will also need estimates for the radius of a cap in terms of its beta measure. The following lemma may be thought of as a ``beta  interpolation" of the corresponding results relating a cap's radius to its normalized surface area   ($\beta=-1$; \cite[Lem. 3.12]{SW2003}) or normalized volume ($\beta=0$). A similar statement was recently shown in \cite[Lem. 2.4]{HLRT-2022} for any continuous positive density on the sphere.

\begin{lemma}\label{beta-caps-lem}
Let $\beta> -1$. For an affine hyperplane $H=H(h)$  in $\R^n$ with distance $h$ from the origin, let $v=\mathbb{P}_\beta(B_n\cap H^-)=\int_{B_n\cap H^-}f_{n,\beta}(x)\,\dint x$ be the beta measure of the cap of $B_n$ cut off by $H^-$.
Set $r=\sqrt{1-h^2}$. Then for all $r \in (0,\frac{3}{4})$,
\begin{align*}
    r^{n+2\beta+1}\leq d_{n,\beta}\, v \leq r^{n+2\beta+1} (1+r^2),
\end{align*}
where
\begin{equation}\label{eqn:d_n_beta}
    d_{n,\beta} := \frac{n+2\beta+1}{c_{1,\beta+\frac{n-1}{2}}} 
    = \frac{2\pi}{B(\frac{1}{2},\frac{n}{2}+\beta+1)}.
\end{equation}
\end{lemma}

\begin{proof}
By Lemma \ref{PC-caps}, the rotational invariance of the beta distribution  
and the substitution $y=1-x^2$, we obtain
\begin{align}\nonumber
v=\mathbb{P}_\beta(B_n\cap H^-)
&=c_{1,\beta+\frac{n-1}{2}}\int_{h}^1 (1-x^2)^{\beta+\frac{n-1}{2}}\,\dint x\\\nonumber
&=\tfrac{1}{2} c_{1,\beta+\frac{n-1}{2}} \int_0^{1-h^2}y^{\beta+\frac{n-1}{2}}(1-y)^{-\frac{1}{2}}\,\dint y\\\label{eqn:proof_eq1}
&=\tfrac{1}{2} c_{1,\beta+\frac{n-1}{2}} \, B\!\left(1-h^2; \tfrac{n+2\beta+1}{2}, \tfrac{1}{2}\right)
\end{align}
where $B(z;a,b)=\int_0^z y^{a-1}(1-y)^{b-1}\,\dint y$ is the incomplete beta function. 
Using $1-h^2=r^2$, for $0\leq r\leq \frac{3}{4}$ we obtain the estimate
\begin{equation*}
    1\leq \frac{1}{\sqrt{1-y}} \leq 1+r^2 \qquad \text{for all $0\leq y \leq r^2$.}
\end{equation*}
This yields
\begin{equation}\label{eqn:proof_eq2}
    \frac{2}{n+2\beta+1} r^{n+2\beta+1} 
    \leq B\!\left(r^2;\tfrac{n+2\beta+1}{2}, \tfrac{1}{2}\right)
    \leq \frac{2}{n+2\beta+1} r^{n+2\beta+1} (1+r^2).
\end{equation}
The statement of the lemma follows by combining \eqref{eqn:proof_eq1} and \eqref{eqn:proof_eq2}.
\end{proof}


\subsection{Asymptotics of the expected volume of a random beta polytope}

The next ingredient is a special case of a result of Affentranger \cite[Cor.\ 1]{Affentranger1991}, which describes the asymptotic behavior of the expected volume of a random beta polytope. In what follows, let $B(v,w)=\int_0^1 t^{v-1}(1-t)^{w-1}\,\dint t$ denote the beta function.

\begin{lemma}\label{beta-volume}
Let $n\geq 1$ and let $X_1,\ldots,X_N$ be i.i.d. random points chosen from $B_n$ according to the beta distribution with $\beta> -1$, and set $P_{n,N}^\beta:=[X_1,\ldots,X_N]$. Then the expected volume of $P_{n,N}^\beta$ satisfies
\begin{align*}\label{beta-volume}
    \lim_{N\to\infty}N^{\frac{2}{n+2\beta+1}}\E[\vol_n(B_n\setminus P_{n,N}^\beta)] = A_{n,\beta},
\end{align*}
where
\begin{equation}\label{eqn:A_n_beta}
A_{n,\beta}:=\frac{\omega_n}{2}\, \frac{n+2\beta+1}{n+2\beta+3} \, \frac{\Gamma\left(n+1+\frac{2}{n+2\beta+1}\right)}{\Gamma(n+1)}\, 
    d_{n,\beta}^{\frac{2}{n+2\beta+1}}.
\end{equation}
\end{lemma}
In  Appendix \ref{sec-appendix}, we show that for every $n\geq 1$ and every $\beta\geq -1/2$,
\begin{equation*}
    A_{n,\beta} = \frac{\omega_n}{2}\left(1+O\left(\frac{\ln(n+2\beta+2)}{n+2\beta+1}\right)\right). 
\end{equation*}
Affentranger \cite[Cor.\ 1]{Affentranger1991} stated his results for $n\geq 2$. In Appendix \ref{sec:appendix-B}, we briefly verify that his arguments are also true in the case $n=1$.

\section{Proof of Theorem \ref{mainThm}}\label{sec-proof-mainThm}

\subsection{Step 1: Reduction to the weighted  symmetric volume difference}

In the first step, we reduce the problem to estimating the expected symmetric volume difference of the projection of a random beta polytope and a Euclidean ball.

\begin{lemma}\label{weak-dist-lem}
Let $U_1,\ldots,U_N$ be chosen independently and uniformly from the sphere $\mathbb{S}^{n-1}$, and set $P_{n,N}^{\rm unif}:=[U_1,\ldots,U_N]$. Then for  any fixed $r>0$ and all $j\in[n]$,
\begin{equation*}
    \E[\delta_j(P_{n,N}^{\rm unif},rB_n)]=\bbinom{n}{j}\E[\vol_j(P_{j,N}^{\beta=\frac{n-j-2}{2}}\triangle rB_j)].
\end{equation*}
\end{lemma}

\begin{proof}
For $\beta>-1$, let $\mu_{n,\beta}$ denote the beta distribution on $B_n$. As shown in the proof of \cite[Cor.\ 3.9]{KabluchkoEtAl2019}, $\mu_{n,\beta}\to\sigma$ weakly as $\beta\to -1^+$, and thus for each $N\in\mathbb{N}$ the product measure $\otimes_{i=1}^N\mu_{n,\beta}$ converges weakly to $\otimes_{i=1}^N\sigma$ as $\beta\to -1^+$. Note that the functional $(x_1,\ldots,x_N)\mapsto \delta_j([x_1,\ldots,x_N],rB_n)$ is continuous on $\prod_{i=1}^N B_n$. Thus by the continuous mapping theorem, the random variable $\delta_j(P_{n,N}^\beta,rB_n)$ converges in distribution to $\delta_j(P_{n,N}^{\rm unif},rB_n)$ as $\beta\to -1^+$. Since these random variables are bounded from above by $(1+r^j) V_j(B_n)$, we obtain 
\[
\lim_{\beta\to -1^+}\E[\delta_j(P_{n,N}^\beta,rB_n)]=\E[\delta_j(P_{n,N}^{\rm unif},rB_n)].
\]
A similar argument  shows that for any $r>0$,
\[
\lim_{\beta\to -1^+}\E[\vol_j( P_{j,N}^{\beta+\frac{n-j}{2}} \triangle rB_j)]=                 \E[\vol_j(P_{j,N}^{\frac{n-j-2}{2}}\triangle rB_j)].
\]

Next, for any subspace $H\in\Gr(n,j)$, Lemma \ref{KTT-proj1} yields 
\[
P_{n,N}^\beta|H \stackrel{d}{=} P_{j,N}^{\beta+\frac{n-j}{2}}
\]
where $\stackrel{d}{=}$ indicates equality in distribution. Thus by Fubini's theorem  and the rotational invariance of the beta distribution, for any $r\in(0,1)$ we have 
\begin{align}\label{reduction-step}
    \E[\delta_j(P_{n,N}^{\rm unif},rB_n)] 
        &= \lim_{\beta\to -1^+} \E[\delta_j(P_{n,N}^\beta,rB_n)]\nonumber\\
        &= \lim_{\beta\to -1^+} \bbinom{n}{j} \int_{\Gr(n,j)} \E[\vol_j((P_{n,N}^\beta|H)\triangle(rB_n|H))]\, \dint \nu_j(H) \nonumber\\ 
        &= \lim_{\beta\to -1^+} \bbinom{n}{j} \int_{\Gr(n,j)} \E[\vol_j( P_{j,N}^{\beta+\frac{n-j}{2}} \triangle rB_j)]\,\dint \nu_j(H) \nonumber\\
        &= \lim_{\beta\to -1^+} \bbinom{n}{j}                 \E[\vol_j( P_{j,N}^{\beta+\frac{n-j}{2}} \triangle rB_j)] \nonumber\\
        &=                      \bbinom{n}{j}                 \E[\vol_j(P_{j,N}^{\frac{n-j-2}{2}}\triangle rB_j)].
\end{align}
\end{proof}

\subsection{Step 2: The choice of scaling factor}

For any $r\in(0,1)$,
\begin{equation}\label{eqn:step2_eqn}
    \E[\vol_n(rB_n\triangle P_{n,N}^\beta)]=\vol_n(B_n\setminus rB_n)-\E[\vol_n(B_n\setminus P_{n,N}^\beta)]+2\E[\vol_n(rB_n\cap (P_{n,N}^\beta)^c)].
\end{equation}
Given $N\geq n+1$ and $\beta\geq -1$, there exists $\gamma_{n,N,\beta}\in(0,1)$ such that 
\begin{equation}\label{choice-gamma}
    \vol_n(B_n\setminus(1-\gamma_{n,N,\beta})B_n) = \E[\vol_n(B_n\setminus P_{n,N}^\beta)].
\end{equation}
Setting $t_{n,N,\beta}:=1-\gamma_{n,N,\beta}$ and $\dint\P_\beta(x)=f_{n,\beta}(x)\,\dint x$ and combining \eqref{eqn:step2_eqn} and \eqref{choice-gamma},  we have
\begin{equation}\label{reduction-1}
\E[\vol_n(t_{n,N,\beta}B_n\triangle P_{n,N}^\beta)]=2\int_{B_n}\cdots\int_{B_n}\vol_n(t_{n,N,\beta} B_n \setminus [x_1,\ldots,x_N])\,\dint\P_\beta(x_1)\cdots \dint\P_\beta(x_N).
\end{equation}
By \eqref{choice-gamma}, Lemma \ref{beta-volume} and the homogeneity of volume, 
\begin{equation}\label{gamma-asymptotics}
    \gamma_{n,N,\beta} \sim \frac{\E[\vol_n(B_n\setminus P_{n,N}^\beta)]}{n\vol_n(B_n)}
    \sim  \frac{A_{n,\beta}}{\omega_n}N^{-\frac{2}{n+2\beta+1}}
\end{equation}
as $N\to\infty$. In particular, for all $\varepsilon>0$ there exists $N_0$ such that for all $N\geq N_0$,
\begin{equation}\label{gamma-bound}
(1-\varepsilon) \frac{A_{n,\beta}}{\omega_n} N^{-\frac{2}{n+2\beta+1}} \leq \gamma_{n,N,\beta} \leq (1+\varepsilon) \frac{A_{n,\beta}}{\omega_n} N^{-\frac{2}{n+2\beta+1}}. 
\end{equation}
Hence by Stirling's inequality, there are positive absolute constants $c_1$ and $c_2$ such that
\begin{equation}\label{gamma-est}
    c_1 N^{-\frac{2}{n+2\beta+1}}\leq \gamma_{n,N,\beta}\leq c_2 N^{-\frac{2}{n+2\beta+1}}.
\end{equation}
More specifically, 
\begin{equation}\label{gamma-const-est}
\frac{A_{n,\beta}}{\omega_n} \sim c_1 \sim c_2 = \frac{1}{2}\left(1+O\left(\frac{\ln(n+2\beta+2)}{n+2\beta+1}\right)\right)
\end{equation}
as $n\to \infty$. This computation is carried out in Appendix \ref{sec-appendix}.

\subsection{Step 3: Adapting the random construction from \texorpdfstring{\cite{LSW}}{[22]} to beta polytopes}

The next result extends the proof of \cite[Thm. 1]{LSW} from the uniform distribution on the sphere $\mathbb{S}^{n-1}$ ($\beta=-1$) to all beta distributions on $B_n$ with $\beta\geq -\frac{1}{2}$.

\begin{theorem}\label{beta-LSW}
Fix $n\geq 1$ and $\beta\geq -\frac{1}{2}$, and let $P_{n,N}^\beta$ be the convex hull of $N\geq n+1$ random points $X_1,\ldots,X_N$ chosen i.i.d. from the Euclidean unit ball $B_n$ with respect to the beta distribution. Then for all sufficiently large $N$, 
\begin{align*}
    \E[\vol_n(B_n\triangle t_{n,N,\beta}^{-1}P_{n,N}^\beta)]&\leq
    \left(1 + O\left(\frac{\ln(n+2\beta+2)}{n+2\beta+1}\right)\right) \frac{2n\vol_n(B_n)}{n+2\beta+1}  N^{-\frac{2}{n+2\beta+1}}.
     \end{align*}
\end{theorem}
\noindent We prove Theorem \ref{beta-LSW} below and apply it in Step 4 where we will replace $(n,\beta)$ by $(j,\frac{n-j-2}{2})$ for $n\geq 2$ and $j\in [n-1]$. Therefore, in our estimates for the constants in Theorem \ref{beta-LSW} we need good bounds for all $n\geq 1$ as $\beta\to \infty$, so that in Step 4 after the replacement we obtain bounds for all $j\in[n]$ as $n\to\infty$.

\begin{proof}[Proof of Theorem \ref{beta-LSW}.]

Choose i.i.d.\ random points $X_1,X_2,\ldots$ from $B_n$ according to the beta density $f_{n,\beta}$, and for $N\geq n+1$ define the random beta polytope $P_{n,N}^\beta:=[X_1,\ldots,X_N]$ to be the convex hull of $X_1,\ldots,X_N$.  
 
Let $\mathcal{E}_{n,N,\beta}$ denote the event that the origin $o$ lies in the interior of $P_{n,N}^\beta$. Then as in \cite[Lem. 4.3 (ii)]{SW2003}, see also \cite[proof of Cor.\ 1]{Affentranger1991}, 
\begin{align*}
\P(\mathcal{E}_{n,N,\beta}^c) &= \P(\{o\not\in \interior [X_1,\ldots,X_N]\})\\
&=\P_\beta^N(\{(x_1,\ldots,x_N)\in B_n^N:\,o\not\in \interior [x_1,\ldots,x_N]\}) \leq e^{-cN}
\end{align*}
for some positive absolute constant $c=c(n,\beta)$, where $B_n^m = \prod_{i=1}^m B_n$ for $m\in\mathbb{N}$.  Hence, by the law of total expectation,
\begin{align*}
&\E[\vol_n(t_{n,N,\beta}B_n \triangle P_{n,N}^\beta)]\\
&=\E[\vol_n(t_{n,N,\beta}B_n \triangle P_{n,N}^\beta)|\mathcal{E}_{n,N,\beta}]\P(\mathcal{E}_{n,N,\beta})+\E[\vol_n(t_{n,N,\beta}B_n \triangle P_{n,N}^\beta)|\mathcal{E}_{n,N,\beta}^c]\P(\mathcal{E}_{n,N,\beta}^c)\\
&\leq \E[\vol_n(t_{n,N,\beta}B_n\triangle P_{n,N}^\beta)|\mathcal{E}_{n,N,\beta}]\P(\mathcal{E}_{n,N,\beta}) + \vol_n(B_n) e^{-cN}.
    \end{align*}
The second term is exponentially decreasing in $N$, while, as we will see, the first is of the order $N^{-\frac{2}{n+2\beta+1}}$. Therefore, the second term is negligible.

By \eqref{reduction-1} and the fact that  $P_{n,N}^\beta$ is simplicial with probability one, we have that
\begin{align*}
    &\E[\vol_n(t_{n,N,\beta}B_n\triangle P_{n,N}^\beta)|\mathcal{E}_{n,N,\beta}] \P(\mathcal{E}_{n,N,\beta})\\
    &\quad =2\int_{B_n}\cdots\int_{B_n} \vol_n(t_{n,N,\beta}B_n\setminus [x_1,\ldots,x_N])\mathbbm{1}_{E_{n,N,\beta}}(x_1,\ldots,x_N)\,\dint \P_\beta(x_1)\cdots \dint \P_\beta(x_N),
\end{align*}
where
\begin{equation*}
    E_{n,N,\beta} := \{(x_1,\dotsc,x_N) \in B_n^N : o \in \interior [x_1,\dotsc,x_N] \text{ and $[x_1,\dotsc,x_n]$ is simplicial}\}.
\end{equation*}
For points $(x_1,\dotsc,x_N)\in E_{n,N,\beta}$, we can decompose $\R^n$ into the following union of cones with pairwise disjoint interiors: 
\begin{equation*}
\R^n = \bigcup_{[x_{j_1},\ldots,x_{j_n}] \in \mathcal{F}_{n-1}([x_1,\ldots,x_N])} {\rm cone}(x_{j_1},\ldots,x_{j_n}),
\end{equation*}
where, for a polytope $P$ in $\R^n$, we let $\mathcal{F}_{n-1}(P)$ denote the set of facets ($(n-1)$-dimensional faces) of $P$, and ${\rm cone}(y_1,\ldots,y_m):=\{\sum_{i=1}^m a_i y_i:\, a_i\geq 0,\, i\in[m]\}$ denotes the cone spanned by a set of vectors $y_1,\ldots,y_m\in\R^n$.

For points $y_1,\ldots,y_n\in\R^n$ such that their affine hull is an $(n-1)$-dimensional hyperplane $H(y_1,\ldots,y_n)$, let $H^+(y_1,\ldots,y_n)$ denote the halfspace bounded by $H(y_1,\ldots,y_n)$ with $o\in H^+(y_1,\ldots,y_n)$.
For  points $x_1,\ldots,x_N\in\R^n$ and a subset of indices $\{j_1,\ldots,j_n\}\subset[N]$, we define a  functional $\Phi^\beta_{j_1,\ldots,j_n}:(\R^n)^N\to[0,\infty)$ by 
\begin{equation*}
    \Phi^\beta_{j_1,\ldots,j_n}(x_1,\ldots,x_N) :=
        \vol_n(t_{n,N,\beta} B_n\cap H^-(x_{j_1},\dotsc,x_{j_n}) \cap \mathrm{cone}(x_{j_1},\dotsc,x_{j_n})),
\end{equation*}
if $[x_1,\dotsc,x_N]$ contains the origin in its interior and $\dim([x_{j_1},\dotsc,x_{j_n}])=n-1$, and set $\Phi^\beta_{j_1,\ldots,j_n}(x_1,\ldots,x_N) :=0$ otherwise.
Then
\begin{align*}
    &\E[\vol_n(t_{n,N,\beta}B_n\triangle P_{n,N}^\beta)|\mathcal{E}_{n,N,\beta}] \P(\mathcal{E}_{n,N,\beta})\\
    &\quad = 2\int_{B_n}\cdots\int_{B_n}\vol_n(t_{n,N,\beta} B_n\setminus [x_1,\ldots,x_N])\mathbbm{1}_{E_{n,N,\beta}}(x_1,\ldots,x_N)\,\dint\P_\beta(x_1)\cdots \dint\P_\beta(x_N)\\
    &\quad=2\int_{B_n}\cdots\int_{B_n} \sum_{\{j_1,\ldots,j_n\}\subset[N]}\Phi^\beta_{j_1,\ldots,j_n}(x_1,\ldots,x_N)\,\dint\P_\beta(x_1)\cdots \dint\P_\beta(x_N)\\
    &\quad=2\binom{N}{n}\int_{B_n}\cdots\int_{B_n} \Phi^\beta_{1,\ldots,n}(x_1,\ldots,x_N)\,\dint\P_\beta(x_1)\cdots \dint\P_\beta(x_N).
\end{align*}
We apply the affine Blaschke--Petkantschin formula in Lemma \ref{BPext} to derive
\begin{align*}
    &\E[\vol_n(t_{n,N,\beta}B_n\triangle P_{n,N}^\beta)|\mathcal{E}_{n,N,\beta}] \P(\mathcal{E}_{n,N,\beta})\\
    &\quad = 2\omega_n(n-1)! \binom{N}{n} \int_{\S^{n-1}} \int_{0}^1 \int_{B_n\cap H} \cdots \int_{B_n\cap H} \\
    &\qquad \qquad 
        \left[\int_{B_n}\cdots \int_{B_n} \Phi^\beta_{1,\ldots,n}(x_1,\ldots,x_N) \, \dint\P_\beta(x_{n+1}) \cdots \dint\P_\beta(x_N)\right] \times\\
    &\qquad \qquad \qquad \qquad 
        \times\vol_{n-1}([x_1,\dotsc,x_n]) \, \dint \P_\beta^H(x_1)\cdots \dint\P_\beta^H(x_n) \, \dint h \, \dint \sigma(u).  
\end{align*}
Now, by definition $\Phi_{1,\dotsc,n}(x_1,\dotsc,x_N)\neq 0$ if $x_1,\dotsc,x_n$ spans a hyperplane $H$ and $[x_1\dotsc,x_n]\in \mathcal{F}_{n-1}([x_1,\dotsc,x_N])$ and $o\in \interior [x_1,\dotsc,x_N]$. In this case, the value of $\Phi_{1,\dotsc,n}(x_1,\dotsc,x_N)$ only depends on $x_1,\dotsc,x_n$. Since 
\begin{align*}
&\big\{(x_{n+1},\ldots,x_N)\in B_n^{N-n}:[x_1,\ldots,x_n]\in\mathcal{F}_{n-1}([x_1,\dotsc,x_N])\text{ and }o\in \interior [x_1,\dotsc,x_N]\big\}\\
&\; \subset \left\{(x_{n+1},\ldots,x_N)\in B_n^{N-n} : x_{n+1},\ldots,x_N\in H^+(x_1,\ldots,x_n)\right\} = (B_n\cap H^+(x_1,\dotsc,x_n))^{N-n},
\end{align*}
we find that
\begin{align}
      \P_\beta^{N-n}\big(\{(x_{n+1},\ldots,x_N)\in B_n^{N-n}&:[x_1,\ldots,x_n]\in\mathcal{F}_{n-1}([x_1,\dotsc,x_N])\text{ and }o\in \interior [x_1,\dotsc,x_N]\}\big) \nonumber\\
      &\leq \mathbb{P}_\beta(E_h^p)^{N-n}=F_{1,\beta+\frac{n-1}{2}}(h)^{N-n}, \label{prob-correct}
\end{align}
where $h=\dist(o,H(x_1,\dotsc,x_n))$.
Here we also used Lemma \ref{PC-caps} and the rotational invariance of the beta distribution.
Thus
\begin{align*}
    &\int_{B_n}\cdots \int_{B_n} \Phi^\beta_{1,\ldots,n}(x_1,\ldots,x_N) \, \dint\P_\beta(x_{n+1}) \cdots \dint\P_\beta(x_N)\\
    &\quad \leq \vol_n(t_{n,N,\beta}B_n \cap H^-(x_1,\dotsc,x_n) \cap \mathrm{cone}(x_1,\dotsc,x_n)) F_{1,\beta+\frac{n-1}{2}}(h)^{N-n}.
\end{align*}
This yields
\begin{align*}
    &\E[\vol_n(t_{n,N,\beta}B_n\triangle P_{n,N}^\beta)|\mathcal{E}_{n,N,\beta}] \P(\mathcal{E}_{n,N,\beta})\\
    &\quad \leq 2\omega_n(n-1)!  \binom{N}{n}\int_{\S^{n-1}}\int_0^1\int_{B_n\cap H}\cdots\int_{B_n\cap H}\vol_n(t_{n,N,\beta}B_n\cap H^-\cap \mathrm{cone}(x_1,\ldots,x_n))\\    
    &\qquad \qquad \times\vol_{n-1}([x_1,\ldots,x_n]) F_{1,\beta+\frac{n-1}{2}}(h)^{N-n}\, \dint\P_\beta^H(x_1)\cdots \dint \P_\beta^H(x_n)\,\dint h\,\dint \sigma(u),
\end{align*}
where the hyperplane $H=H(u,h)$ has distance $h$ from the origin and is orthogonal to $u\in\S^{n-1}$.

Following as in \cite[page 8]{LSW}, for all sufficiently large $N$ we estimate, for $0 \leq h \leq \frac{n+2\beta+2}{n+2\beta+3}$, 
\begin{equation*}
F_{1,\beta+\frac{n-1}{2}}(h)^{N-n}=\left(1-\mathbb{P}_\beta(B_n\cap H^-)\right)^{N-n}\leq e^{-\mathbb{P}_\beta(B_n\cap H^-)(N-n)}.
\end{equation*}
Note that for $0\leq h\leq \frac{n+2\beta+2}{n+2\beta+3}$ we have
\begin{equation*}
    \mathbb{P}_\beta(B_n\cap H^-(u,h)) \geq \mathbb{P}_\beta\left(B_n\cap H^-\left(u,\frac{n+2\beta+2}{n+2\beta+3}\right)\right) > 0.
\end{equation*}
Consequently, for all sufficiently large $N$, 
\begin{align*}
    &2\omega_n(n-1)!  \binom{N}{n}\int_{\S^{n-1}}\int_0^{\frac{n+2\beta+2}{n+2\beta+3}}\int_{B_n\cap H}\cdots\int_{B_n\cap H}\\
    &\qquad \times \vol_n(t_{n,N,\beta}B_n\cap H^-\cap \mathrm{cone}(x_1,\ldots,x_n))
    \vol_{n-1}([x_1,\ldots,x_n]) F_{1,\beta+\frac{n-1}{2}}(h)^{N-n} \\
    &\qquad\qquad\qquad \times\dint\P_\beta^H(x_1)\cdots \dint \P_\beta^H(x_n)\,\dint h\,\dint \sigma(u)\\
    &\quad\leq 2 \vol_n(B_n)^3 n! \binom{N}{n} e^{- \mathbb{P}_\beta\left(B_n\cap H^-\left(u_0,\frac{n+2\beta+2}{n+2\beta+3}\right)\right) (N-n)}  \leq e^{-C(n,\beta) N}
\end{align*}
for some positive constant $C(n,\beta)$.

Next, observe that
\begin{equation*}
    t_{n,N,\beta}B_n\cap H^-(u,h)\cap{\rm cone}(x_1,\ldots,x_n) \subset H^+(u,t_{n,N,\beta})\cap H^-(u,h)\cap{\rm cone}(x_1,\ldots,x_n),
\end{equation*}
where the latter object is a truncated cone containing the desired cap. Since the top and bottom facets of the truncated cone are homothets, by the monotonicity and homogeneity of the volume functional,
\begin{equation*}
    \vol_n(t_{n,N,\beta}B_n\cap H^-\cap{\rm cone}(x_1,\ldots,x_n)) \leq \frac{h}{n}\max\left\{0,\left(\frac{t_{n,N,\beta}}{h}\right)^n-1\right\}\vol_{n-1}([x_1,\ldots,x_n]).
\end{equation*}
Therefore, using also the fact that the beta distribution is rotationally invariant, we obtain
\begin{align*}
    &\E[\vol_n(t_{n,N,\beta}B_n\triangle P_{n,N}^\beta)] \\
    &\quad \leq   2\omega_n(n-1)! \binom{N}{n}\int_{\mathbb{S}^{n-1}}\int_{\frac{n+2\beta+2}{n+2\beta+3}}^1\int_{B_n\cap H}\cdots\int_{B_n\cap H} \frac{h}{n}\max\left\{0,\left(\frac{t_{n,N,\beta}}{h}\right)^{\!n}-1\right\}\\
    &\qquad \times F_{1,\beta+\frac{n-1}{2}}(h)^{N-n}\, \vol_{n-1}([x_1,\ldots,x_n])^2\, \dint\P_\beta^H(x_1)\cdots \dint\P_{\beta}^H(x_n)\,\dint h\,\dint\sigma(u) + e^{-C(n,\beta)N}\\
    &\quad =   2\omega_n(n-1)!\binom{N}{n}\int_{\frac{n+2\beta+2}{n+2\beta+3}}^1  \frac{h}{n}\max\left\{0,\left(\frac{t_{n,N,\beta}}{h}\right)^n-1\right\}F_{1,\beta+\frac{n-1}{2}}(h)^{N-n}\\
    &\qquad \times\int_{B_n\cap H}\cdots\int_{B_n\cap H}\vol_{n-1}([x_1,\ldots,x_n])^2\, \dint\P_{\beta}^H(x_1) \cdots \dint\P_{\beta}^H(x_n)\,\dint h + e^{-C(n,\beta)N},
\end{align*}
where in the last line we fixed an arbitrary direction $u_0\in\S^{n-1}$ and $H=H(u_0,h)$ is at distance $h$ from the origin. By Lemma \ref{2nd-moment} and the equation $r^2+h^2=1$, we have 
\begin{align*}
     &\E[\vol_n(t_{n,N,\beta}B_n\triangle P_{n,N}^\beta)] \\
     &\quad \leq 2\omega_n(n-1)!\frac{c_{n,\beta}^n}{c_{n-1,\beta}^n}\E[\mathcal{V}_{n,n-1}^2]\times\\
     &\qquad \times\binom{N}{n}\int_{\frac{n+2\beta+2}{n+2\beta+3}}^1 r^{2n\beta+(n-1)(n+2)}\frac{h}{n}\max\left\{0,\left(\frac{t_{n,N,\beta}}{h}\right)^n-1\right\}F_{1,\beta+\frac{n-1}{2}}(h)^{N-n} \,\dint h + e^{-C(n,\beta)N}.
\end{align*}
Since $h\geq \frac{n+2\beta+2}{n+2\beta+3}$ and $1-t_{n,N,\beta}=\gamma_{n,N,\beta}$ is of the order $N^{-\frac{2}{n+2\beta+1}}$,  for all sufficiently large $N$ it holds true that
\begin{align*}
    \frac{1}{n}\left[\left(\frac{t_{n,N,\beta}}{h}\right)^n-1\right] 
    &= \frac{t_{n,N,\beta} - h}{nh} \left[ 1+ \frac{t_{n,N,\beta}}{h} + \dotsc +\left(\frac{t_{n,N,\beta}}{h}\right)^{n-1}\right]\\
    &\leq \frac{t_{n,N,\beta} - h}{nh} \left[ 1+ \frac{1}{h} + \dotsc +\frac{1}{h^{n-1}}\right]\\
    &\leq \frac{t_{n,N,\beta} - h}{n} \frac{n+2\beta+3}{n+2\beta+2} \left[ 1+ \frac{n+2\beta+3}{n+2\beta+2} + \dotsc + \left(\frac{n+2\beta+3}{n+2\beta+2}\right)^{n-1}\right]\\
    &= \left[\left(1+\frac{1}{n+2\beta+2}\right)^n-1\right] \frac{n+2\beta+3}{2n} (t_{n,N,\beta} - h).
\end{align*}
Since
\begin{equation*}
    \left(1+\frac{1}{n+2\beta+2}\right)^n \leq 1+\frac{2n}{n+2\beta+3} \qquad \text{for all $n\geq 1$ and $\beta\geq -1$,}
\end{equation*}
this yields
\begin{equation*}
    \frac{1}{n}\left[\left(\frac{t_{n,N,\beta}}{h}\right)^n-1\right]\leq t_{n,N,\beta}-h.
\end{equation*}

Hence, for all sufficiently large $N$,
\begin{align}\label{shorter-proof}
     \E[\vol_n(t_{n,N,\beta}B_n\triangle P_{n,N}^\beta)]
    &\leq 2\omega_n(n-1)!\frac{c_{n,\beta}^n}{c_{n-1,\beta}^n}\E[\mathcal{V}_{n,n-1}^2] \binom{N}{n}\nonumber\\
        &\times \int_{\frac{n+2\beta+2}{n+2\beta+3}}^{t_{n,N,\beta}} 
        r^{2n\beta+(n-1)(n+2)}(t_{n,N,\beta}-h)(1-\phi_{\beta}(h))^{N-n} \,\dint h + e^{-C(n,\beta)N},
\end{align}
where we set
\begin{equation*}
    \phi_\beta(h):=1-F_{1,\beta+\frac{n-1}{2}}(h)=\mathbb{P}_\beta(B_n\cap H^-(u_0,h)) \qquad \text{for $h\in[0,1]$.}
\end{equation*}
By Lemma \ref{beta-caps-lem} we have that $r^{n+2\beta+1} \leq d_{n,\beta} \phi_{\beta}(h)$. This, together with the inequality
\begin{equation*}
    t_{n,N,\beta}-h\leq 1-h = 1-\sqrt{1-r^2} \leq r^2,
\end{equation*}
yields
\begin{equation*}
    \E[\vol_n(t_{n,N,\beta}B_n\triangle P_{n,N}^\beta)]
     \leq \frac{2n\omega_n}{d_{n,\beta}^{n-1}} \binom{N}{n} \int_{\frac{n+2\beta+2}{n+2\beta+3}}^{1} (d_{n,\beta} \phi_\beta(h))^{n} (1-\phi_\beta(h))^{N-n} \,\dint h + e^{-C(n,\beta)N}.
\end{equation*}
Here we used the fact that
\begin{equation}
    (n-1)!\frac{c_{n,\beta}^n}{c_{n-1,\beta}^n}\E[\mathcal{V}_{n,n-1}^2]
    = d_{n,\beta}^{-(n-1)},
\end{equation}
which follows from Lemma \ref{beta-Moments}. 

We now use the change of variable $t = N\phi_{\beta}(h)$ and note that by Lemmas \ref{beta-caps-lem} and \ref{caps-lem},
\begin{align*}
    \phi_{\beta}'(h) &= - \frac{n+2\beta+1}{d_{n,\beta}} r^{n+2\beta-1} 
    \leq -\frac{n+2\beta+1}{d_{n,\beta}} \left(\frac{d_{n,\beta} \phi_{\beta}(h)}{2-h^2}\right)^{1-\frac{2}{n+2\beta+1}}\\
    &\leq -\frac{n+2\beta+1}{d_{n,\beta}} (d_{n,\beta} \phi_{\beta}(h))^{1-\frac{2}{n+2\beta+1}} \left(\frac{n+2\beta+3}{n+2\beta+5-\frac{1}{n+2\beta+3}}\right)^{1-\frac{2}{n+2\beta+1}}\\
    &\leq -\frac{n+2\beta+1}{d_{n,\beta}} \, \frac{n+2\beta+3}{n+2\beta+4}\,  (d_{n,\beta} \phi_{\beta}(h))^{1-\frac{2}{n+2\beta+1}},
\end{align*}
where we used the fact that $h\geq \frac{n+2\beta+2}{n+2\beta+3}$.
This yields
\begin{align*}
    &\E[\vol_n(t_{n,N,\beta}B_n\triangle P_{n,N}^\beta)]\\
    &\qquad \leq \left(1+\frac{1}{n+2\beta+3}\right) \, \frac{2n\omega_n}{n+2\beta+1} \, \frac{1}{N^n}\binom{N}{n} \left(\frac{d_{n,\beta}}{N}\right)^{\frac{2}{n+2\beta+1}} \\
    &\qquad \qquad \times \int^{N\phi_{\beta}(\frac{n+2\beta+2}{n+2\beta+3})}_{0}
        t^{n-1+\frac{2}{n+2\beta+1}} \left(1-\frac{t}{N}\right)^{N-n} \,\dint t + e^{-C(n,\beta)N}.
\end{align*}
Note that
\begin{align*}
    \left(1-\frac{t}{N}\right)^{-n} 
    \leq \left(1+2\phi_\beta\left(\frac{n+2\beta+2}{n+2\beta+3}\right)\right)^n
    &\leq \left(1+ C_1(n,\beta) n \left(\frac{2}{n+2\beta+3}\right)^{\frac{n+2\beta+1}{2}}\right)\\
    &\leq 1+e^{- C_2(n,\beta) (n+2\beta+1)}
\end{align*}
for some positive constants $C_1(n,\beta)$ and $C_2(n,\beta)$. Therefore,
\begin{equation*}
    \int^{N\phi_{\beta}(\frac{n+2\beta+2}{n+2\beta+3})}_0
        t^{n-1+\frac{2}{n+2\beta+1}} \left(1-\frac{t}{N}\right)^{N-n} \,\dint t
    \leq (1+e^{-O(n+2\beta+1)}) \Gamma\left(n+\frac{2}{n+2\beta+1}\right).
\end{equation*}
From this we conclude that
\begin{align*}
    &\E[\vol_n(t_{n,N,\beta}B_n\triangle P_{n,N}^\beta)]\\
    &\quad \leq  \left(1+O\left((n+2\beta+3)^{-1}\right)\right) \frac{2\omega_n}{n+2\beta+1} \frac{n!}{N^n}\binom{N}{n}
        \frac{\Gamma\left(n+\frac{2}{n+2\beta+1}\right)}{\Gamma(n)} \left(\frac{d_{n,\beta}}{N}\right)^{\!\frac{2}{n+2\beta+1}} \!\!+ e^{-C(n,\beta)N}.
\end{align*}
A classical inequality of Wendel \cite{wendel} states that for $x>0$ and $s\in(0,1)$, we have $1\geq \frac{\Gamma(x+s)}{\Gamma(x)x^s}\geq \left(\frac{x}{x+s}\right)^{1-s}$. Applying this with $x=n$ and $s=\frac{2}{n+2\beta+1}$, we obtain
\begin{equation*}
    \frac{\Gamma\left(n+\frac{2}{n+2\beta+1}\right)}{\Gamma(n)} \leq n^{\frac{2}{n+2\beta+1}} \leq 1+O\left(\frac{\ln(n+2\beta+2)}{n+2\beta+1}\right).
\end{equation*}
Thus, using the inequality
\begin{equation*}
    \frac{n!}{N^n} \binom{N}{n} \leq 1,
\end{equation*}
we derive that for sufficiently large $N$,
\begin{equation*}
    \E[\vol_n(t_{n,N,\beta}B_n\triangle P_{n,N}^\beta)] 
    \leq   \left(1+O\left(\frac{\ln(n+2\beta+2)}{n+2\beta+1}\right)\right)  \frac{2n\vol_n(B_n)}{n+2\beta+1}  \left(\frac{d_{n,\beta}}{N}\right)^{\frac{2}{n+2\beta+1}} + e^{-C(n,\beta)N}.
\end{equation*}
By \eqref{eqn:dnbeta_ineq} we find that
\begin{equation*}
    (n d_{n,\beta})^{\frac{2}{n+2\beta+1}}\leq 1 + O\left(\frac{\ln(n+2\beta+2)}{n+2\beta+1}\right).
\end{equation*}
Also, by \eqref{gamma-bound}, we have that
\begin{equation*}
    t_{n,N,\beta}^{-n} \leq \left(1-c_2 N^{-\frac{2}{n+2\beta+1}}\right)^{-n} 
    \leq 1+ 2c_2 n N^{-\frac{2}{n+2\beta+1}} 
    \leq 1+O\left(\frac{\ln(n+2\beta+2)}{n+2\beta+1}\right)
\end{equation*}
for large enough $N$. 
Hence, for all sufficiently large $N$,
\begin{align*}
    \E[\vol_n(B_n\triangle t_{n,N,\beta}^{-1} P_{n,N}^\beta)]
    &= t_{n,N,\beta}^{-n}\E[\vol_n( t_{n,N,\beta}B_n\triangle P_{n,N}^\beta)]\\
    &\leq  \left(1 + O\left(\frac{\ln(n+2\beta+2)}{n+2\beta+1}\right)\right) \frac{2n\vol_n(B_n)}{n+2\beta+1}  N^{-\frac{2}{n+2\beta+1}},
\end{align*}
where we choose $N$ large enough so that $N^{\frac{2}{n+2\beta+1}} e^{-C(n,\beta)N} \leq \frac{2n\vol_n(B_n)}{(n+2\beta+1)^2}$. This completes the proof of Theorem \ref{beta-LSW}. 
\end{proof}

\subsection{Step 4: Going from the local estimates to the global estimate}

To conclude the proof of Theorem \ref{mainThm}, we apply Lemma \ref{weak-dist-lem} as follows. We use the upper bound from Theorem \ref{beta-LSW}, first replacing $n$ by $j$, and then selecting the parameter  $\beta=\frac{n-j-2}{2}$ which corresponds to a $j$-dimensional projection of the uniform distribution on $\mathbb{S}^{n-1}$ ($\beta=-1$). The choice of the scaling factor $r\in(0,1)$ we use in \eqref{reduction-step} is the number $r=t_{j,N,\frac{n-j-2}{2}}$. Substituting this choice of parameters into the upper bound for the expectation in Theorem \ref{beta-LSW} and using the identity  $V_j(B_n)=\tbbinom{n}{j}\vol_j(B_j)$, we obtain that for all  sufficiently large $N$,
\begin{align*}
  \E[\delta_j(t_{j,N,\frac{n-j-2}{2}}^{-1}P_{n,N}^{\rm unif},B_n)]
    &=\bbinom{n}{j}\E[\vol_j(t_{j,N,\frac{n-j-2}{2}}^{-1}P_{j,N}^{\beta=\frac{n-j-2}{2}}\triangle B_j)]\\
    &\leq \left(1 + O\left(\frac{\ln n}{n-1}\right)\right) \bbinom{n}{j}\frac{2j\vol_j(B_j)}{n-1}\NN\\
    &=\left(1 + O\left(\frac{\ln n}{n-1}\right)\right) \frac{2j}{n-1}V_j(B_n)\NN.
\end{align*}
To conclude, we use the fact that for all sufficiently large $N$,  there exists a realization $Q_{n,j,N}^*$ of the random polytope $t_{j,N,\frac{n-j-2}{2}}^{-1}P_{n,N}^{\rm unif}$ which also achieves this upper bound on the expectation. This completes the proof of Theorem \ref{mainThm}. \qed

\section*{Acknowledgments}

SH would like to thank  the Georgia Institute of Technology and the organizers of the {\it Workshop in Convexity and High-Dimensional Probability}, held in May of 2022, for their hospitality (supported by NSF CAREER DMS-1753260). During the  workshop,   significant progress on this manuscript was achieved. We are very grateful to the anonymous referee for detailed comments that helped improve the paper. 


\appendix

\section{Estimation of the constants in \texorpdfstring{\eqref{gamma-const-est}}{(24)}}\label{sec-appendix}
We consider the constant
\begin{equation*}
    \frac{A_{n,\beta}}{\omega_n} = \frac{1}{2} \, \frac{n+2\beta+1}{n+2\beta+3}\,\frac{\Gamma\left(n+1+\frac{2}{n+2\beta+1}\right)}{\Gamma(n+1)} \, d_{n,\beta}^{\frac{2}{n+2\beta+1}}
\end{equation*}
that appears in the estimates \eqref{gamma-bound} for $\gamma_{n,N,\beta}$. A classical inequality by Wendel \cite{wendel}  states that for every $x>0$ and every $s\in(0,1)$ we have $x^s \geq \frac{\Gamma(x+s)}{\Gamma(x)}$. Hence, 
\begin{equation*}
    1 \leq \frac{\Gamma\left(n+1+\frac{2}{n+2\beta+1}\right)}{\Gamma(n+1)} \leq (n+1)^{\frac{2}{n+2\beta+1}} = 1+O\left(\frac{\ln(n+1)}{n+2\beta+1}\right)
\end{equation*}
for all $n\geq 1$ and $\beta\geq -\frac{1}{2}$.

Applying Wendel's inequality again with $x=\frac{n}{2}+\beta+1$ and $s=\frac{1}{2}$, we get
\begin{align*}
    \frac{\Gamma(\frac{n+1}{2}+\beta+1)}{\Gamma(\frac{n}{2}+\beta+1)} \leq \sqrt{\frac{n+2\beta+2}{2}},
\end{align*}
which yields
\begin{align}\nonumber
    1 \leq d_{n,\beta}^{\frac{2}{n+2\beta+1}} 
    &= \left(\frac{2\sqrt{\pi} \Gamma(\frac{n+1}{2}+\beta+1)}{\Gamma(\frac{n}{2}+\beta+1)}\right)^{\frac{2}{n+2\beta+1}}\\\label{eqn:dnbeta_ineq}
    &\leq \big((2\pi) (n+2\beta+2)\big)^{\frac{1}{n+2\beta+1}}
    \leq 1 + O\left(\frac{\ln(n+2\beta+2)}{n+2\beta+1}\right).
\end{align}
Therefore, for all $n\geq 1$ and all $\beta\geq -\frac{1}{2}$ we have
\begin{equation*}
    \frac{1}{2} \left(1-\frac{C_1}{n+2\beta+3}\right) \leq \frac{A_{n,\beta}}{\omega_n} 
    \leq \frac{1}{2} \left(1+ \frac{C_2 \ln[(n+1)(n+2\beta+2)]}{n+2\beta+1}\right)
\end{equation*}
for some positive absolute constants $C_1$ and $C_2$.

\section{1-dimensional case of Lemma \ref{beta-volume}} \label{sec:appendix-B}

Since Affentranger \cite[Cor.\ 1]{Affentranger1991} only stated his results for $n\geq 2$, in the following we briefly verify the statement of Lemma \ref{beta-volume} also holds for $n=1$. 
Let $X_1,\dotsc,X_N$ be i.i.d.\ random points in $B_1=[-1,1]$ chosen according to the beta distribution with $\beta>-1$, and set $P_{1,N}^\beta:=[X_1,\dotsc,X_N]$. Then
\begin{equation*}
    \vol_1(P_{1,N}^\beta) = \max_{i\in [N]} X_i - \min_{i\in [N]} X_i.
\end{equation*}
Thus by the symmetry of the beta distribution,
\begin{align*}
    \E \vol_1(P_{1,N}^\beta) &= \E \max_{i\in[N]} X_i - \E \min_{i\in[N]} X_i = 2\E\max_{i\in[N]} X_i\\ 
    &= 2 \binom{N}{1} \int_{B_1} x \, \P_\beta([-1,x])^{N-1}\, \dint\P_\beta(x)\\
    &= 2c_{1,\beta} N \int_{0}^2 (1-s)\, F_{1,\beta}(1-s)^{N-1} (1-(1-s)^2)^\beta \,\dint s.
\end{align*}
We also have
\begin{align*}
    F_{1,\beta}(1-s) &= c_{1,\beta} \int_{-1}^{1-s} (1-z^2)^\beta \, \dint z = 1 - c_{1,\beta} \int_{0}^s (1-(1-t)^2)^\beta\,\dint t\\
    &\sim 1- \frac{c_{1,\beta} 2^\beta s^{\beta+1}}{\beta+1} \qquad \text{as $s\to 0^+$}.
\end{align*}
Using the substitution $F_{1,\beta}(1-s) = 1-\frac{t}{N}$, we find that
\begin{equation*}
    s = 1- F_{1,\beta}^{-1}\left(1-\frac{t}{N}\right) \sim \left(\frac{(\beta+1)t}{2^\beta N c_{1,\beta}}\right)^{\frac{1}{\beta+1}} \qquad \text{as $N\to \infty$}.
\end{equation*}
Hence,
\begin{align*}
    \E \vol_1(P_{1,N}^\beta) &\sim 2 \int_{0}^N \left(1-\frac{t}{N}\right)^{N-1}\, \dint t 
        - \left(\frac{2(\beta+1)}{c_{1,\beta}}\right)^{\frac{1}{\beta+1}} N^{-\frac{1}{\beta+1}} \int_{0}^N t^{\frac{1}{\beta+1}} \left(1-\frac{t}{N}\right)^{N-1}\, \dint t \\
    &\sim 2 - d_{1,\beta}^{\frac{1}{\beta+1}}\, \Gamma\!\left(1+\tfrac{1}{\beta+1}\right) N^{-\frac{1}{\beta+1}}
    = \vol_1(B_1) - A_{1,\beta}\, N^{-\frac{1}{\beta+1}} \qquad \text{as $N\to \infty$.}
\end{align*}
This proves the statement of Lemma \ref{beta-volume} for $n=1$ and $\beta>-1$.

\bibliographystyle{plain}
\bibliography{main}


\vspace{3mm}

\noindent {\sc Institute of Discrete Mathematics and Geometry, TU Wien,\\ Austria}

\noindent {\it E-mail address:} {\tt florian.besau@tuwien.ac.at}

\vspace{3mm}

\noindent {\sc Department of Mathematics \& Computer Science, Longwood University,\\ U.S.A.}

\noindent {\it E-mail address:} {\tt hoehnersd@longwood.edu}

\end{document}